\documentclass[12pt,a4paper]{article}%
\usepackage{tikz}
\usepackage[T1]{fontenc}
\usepackage[english]{babel}

\usetikzlibrary{arrows.meta,positioning}
\usepackage{pgfplots}
\usepgfplotslibrary{fillbetween}
\usepackage{svg} 
\usepackage{authblk} 
\pgfplotsset{compat=1.18}

\usepackage{bm}

\usepackage{xcolor}
\usepackage{comment}
\usepackage{graphicx}
\usepackage{caption}
\usepackage{subcaption}

\usepackage{indentfirst}

\usepackage[allcolors=blue,colorlinks=true, pdfstartview=FitH, linkcolor=blue, pdftoolbar=true, bookmarks=true,bookmarksnumbered,plainpages]{hyperref}

\usepackage[lmargin=3.5cm,tmargin=3cm,bmargin=3.5cm,rmargin=3cm]{geometry}

\usepackage{amsmath,amssymb,amsthm,amsfonts,mathtools}
\usepackage{times}
\usepackage{mathrsfs}
\usepackage{txfonts}
\usepackage{siunitx}

\DeclareMathVersion{normal2}

\def\XXint#1#2#3{{\setbox0=\hbox{$#1{#2#3}{\int}$}
		\vcenter{\hbox{$#2#3$}}\kern-.5\wd0}}

\newcommand{\Csharp}{{\settoheight{\dimen0}{C}\kern-.09em \resizebox{!}{\dimen0}{\raisebox{\depth}{$\sharp$}}}}

\newtheoremstyle{theorem}
{5pt +1\p@ -2.0\p@}
{5pt +1\p@ -2.0\p@}
{\it}			      
{}				  
{\bfseries}   
{.}               
{.4em}       
{}
\theoremstyle{theorem}
\newtheorem{theorem}{Theorem}[section]
\newtheorem{definition}[theorem]{Definition}
\newtheorem{proposition}[theorem]{Proposition}
\newtheorem{corollary}[theorem]{Corollary}
\newtheorem{lemma}[theorem]{Lemma}
\newtheorem{remark}[theorem]{Remark}
\numberwithin{equation}{section}

\usepackage{mathptmx}
\usepackage{dsfont} 
\usepackage{bbm}
\usepackage{pifont}

\usepackage{enumitem}

\newcommand{\CCap}{\dot{\textnormal{C}}\textnormal{ap}}
\newcommand{\ccap}{\textnormal{Cap}}
\newcommand{\cdGama}{\check{{\Gamma}}^{\alpha}}

\newcommand{\dWolff}{{\textnormal{\textbf{W}}}_{\alpha,q}}
\newcommand{\Wolff}{{\textnormal{\textbf{W}}}_{\alpha,q}}

\newcommand{\dHM}{{\dot{\textnormal{\textbf{V}}}}_{\alpha,q}}

\newcommand{\dM}{{\dot{\textnormal{\textbf{V}}}}_{\alpha,2}}

\usepackage{tocloft}

\title{Nonlinear parabolic thin sets and parabolic Wolff inequalities}

\author[1]{Marcelo F. de Almeida \thanks{Departamento de Matem\'atica, Universidade Federal de Sergipe, S\~ao Crist\'ov\~ao, SE, 49000-000, Brasil.
email: \texttt{marcelo@mat.ufs.br}. Supported by CNPq - Grants:311321/2021-6. 
}
\quad and \quad
Edilson P. dos Santos Filho \thanks{Departamento de Matem\'atica, Universidade de Bras\'ilia, Bras\'ilia, DF, 70910-900, Brasil. \\
email: \texttt{edp9786@gmail.com}. The second author was supported by Capes.
}}

\date{}

\begin{document}

\maketitle
\begin{abstract}
	We prove a parabolic analogue of Wolff's inequality adapted to the intrinsic scaling $\delta_c(x,t)=(cx,c^2t)$ and formulated in terms of time-backward parabolic  dyadic rectangles. As a consequence, we obtain equivalent characterizations of parabolic $(\alpha,q)$-thinness in this geometric setting and establish the associated Kellogg and Choquet properties. We further use the notion of $(\alpha,2)$-thinness defined in terms of fractional heat balls and prove that the sets of irregular boundary points $z_0\in\partial\Omega$ for the heat operator $\partial_t-\Delta$ and for the degenerate operator $\mathscr{L}a=\partial_t(|y|^a\cdot)-\operatorname{div}(|y|^a\nabla\cdot)$ in $\Omega\subset\mathbb{R}^{d+1}$ are negligible with respect to the thermal capacity $\mathrm{cap}^{\mathcal T}$ and the parabolic Bessel capacity $C_{\alpha,2}$, respectively.


\end{abstract}

\tableofcontents

\setcounter{equation}{0}\setcounter{theorem}{0}

\section{Introduction and main results}
By Perron-Wiener-Brelot method we know that given a continuous function $f$ on $\partial\Omega$, there exists $u_f\in C^{2}(\Omega)$ harmonic in $\Omega$ and uniquely determined by $f$ on boundary $\partial\Omega$, that is, the generalized Dirichlet's problem 
\begin{align}\label{dirichlet}
	\begin{cases} 
		\Delta u = 0\quad\;\;\;   \text{ in } &\Omega\\
		u\vert_{\partial\Omega}=f \quad  \text{on } &\partial\Omega,
	\end{cases} 
\end{align}
is solvable for each prescribed $f\in C(\partial \Omega)$, where $\Omega$ is an arbitrary bounded domain. However, the problem \eqref{dirichlet} is not in general solvable due to the presence of irregular boundary points, as illustrated by the classical Lebesgue spine in $\mathbb{R}^3$. This leads to a natural question: for which boundary points $x\in\partial\Omega$ does the harmonic function $u_f$ satisfy 
\begin{equation}\label{question1}
	\lim_{\stackrel{y\rightarrow x}{y\in\Omega}} u_f(y)=f(x)?
\end{equation}
A boundary point  $x\in\partial\Omega$ is said to be regular, if \eqref{question1} holds for every continuous function $f$ on $\partial\Omega$; otherwise, it is classified as irregular.  In 1924, Wiener \cite{Wiener} solved the question above summarized as follows:
\begin{theorem}[Wiener's Criterion]\label{ww} Let $d\geq 3$ and $\Omega\subset\mathbb{R}^d$ be a Borel set. Then $x \in \partial\Omega$ is a regular point for \eqref{dirichlet} if and only if
	\begin{equation}
		\int_{0}^{1}\dfrac{C_{2}(\Omega^{c}\cap B_{r}(x))}{r^{d-2}}\frac{dr}{r} = \infty,\nonumber
	\end{equation}
where $\Omega^c=\mathbb{R}^d\backslash \Omega$ and $C_2(\cdot):=c_{1,2}(\cdot)$ denotes the Bessel capacity of set $E$ defined by  
\begin{equation}\label{+}
	c_{\alpha,q}(E) = \inf \left\{\int_{\mathbb{R}^{d}}\vert g(x)\vert^{q}dx:\;g\geq 0,\;\;G_{\alpha}\ast g\geq \mathds{1}_E \right\},
\end{equation}
and $G_{\alpha}$ is the so-called Bessel kernel.
\end{theorem}
 The classical Wiener's criterion inspired the introduction of $(1,2)$-thin sets, a concept pioneered by M. Brelot \cite{Brelot} that occupies a central place in potential theory. This notion can now be elegantly formulated as follows: a set $E \subset \mathbb{R}^{d}$ is termed $(\alpha,2)-$thin at $x_{0}$, if either $x_0\notin \overline{E}$ or $x_0\in\overline{E}$ and there exists a positive Radon  measure $\mu\in \mathfrak{M}^{+}$ such that 
\begin{equation}
	G_{2\alpha} \ast \mu(x_{0}) < \liminf_{\substack{x \rightarrow x_{0}\\x \in E \setminus\{x_{0}\}}}G_{2\alpha} \ast \mu(x), \quad 0<\alpha\leq d/2.
\end{equation}
Classical potential theory (see \cite[Thm.~5.10]{Landkof}, \cite[Thm.~6.3.2]{AH}) shows that the following assertions are equivalent:
\begin{itemize}
	\item[(E$_1$)] The set $E\subset \mathbb{R}^d$ is $(\alpha,2)-$thin at $x_0$.
	\item[(E$_2$)] Let $B_r(x_0)$ be an open ball and $x_0\in\mathbb{R}^d$, then 
	\begin{equation}\label{+++} 
		\int_{0}^{1}\dfrac{c_{\alpha,2}(E\cap B_{r}(x_0))}{r^{d-2\alpha}}\frac{dr}{r} < \infty.\nonumber
	\end{equation}
	\item[(E$_3$)] The point $x_0\in\overline{E}$ is irregular. 
\end{itemize} 
This equivalence for Dirichlet problem \eqref{dirichlet} reads as follows: 
\begin{equation}
x_0\in\partial\Omega \text{ is regular} \;\Leftrightarrow\; \int_{0}^{1}\frac{c_{1,2}(\Omega^c\cap B_{r}(x_0))}{r^{d-2}}\frac{dr}{r}=+\infty\; \Leftrightarrow\; \Omega^c \text{ is not } (1,2)-\text{thin at } x_0\nonumber.
\end{equation}

The problem of establishing a precise parabolic analogue of Wiener's criterion 
for the Dirichlet problem in bounded space--time domains 
$\Omega \subset \mathbb{R}^{d+1}$ remained a central open question 
in parabolic potential theory for decades. 
Landis \cite[Theorem 2]{Landis} (see also \cite[Theorem 1.3]{Tralli}, 
where a Wiener--Landis criterion for hypoelliptic operators of H\"ormander type 
is established) addressed boundary regularity for the heat equation, 
but his approach relied on recursive sequences and geometric conditions 
rather than a capacitary framework. A pivotal advancement came from Lanconelli \cite[Theorem 1.2]{Lanconelli},  who introduced parabolic capacity as a rigorous tool, bridging elliptic potential theory  with parabolic PDEs and obtaining partial progress toward a parabolic Wiener criterion.  A complete resolution was achieved by Evans and Gariepy \cite{EG}, who provided a necessary and sufficient condition for the regularity of boundary points for the heat operator. This criterion, analogous to its elliptic predecessor, 
characterizes regularity through a capacity-based condition that measures the 
``thickness'' of the complement $\Omega^c$ near a boundary point in a manner consistent with the intrinsic scaling of parabolic potentials. To place this in context, consider the Dirichlet problem for the heat equation:
\begin{align}\label{diri-parab}
	\begin{cases} 
		\partial_tu-\Delta u = 0\quad \text{ in } \quad  \Omega\\
		\;\;\;\;\;\;\,u\vert_{\partial\Omega}=f \quad \text{ on} \quad\partial\Omega,
	\end{cases} 
\end{align}
where $u=H_f^{\Omega}$ denotes the generalized solution in the sense of Perron-Wiener-Brelot-Bauer, for each continuous function $f:\partial\Omega\rightarrow\mathbb{R}$.  A point $z_0 \in \partial \Omega$ is called regular for the operator $\partial_t-\Delta$ on $\Omega$, if $\lim_{z\rightarrow z_0}H_f^{\Omega}(z)=f(z_0)$ for all $f\in C(\partial\Omega)$.

\begin{theorem}[\cite{Lanconelli, EG}]\label{ev} Let  $\Omega \subset \mathbb{R}^{d+1}$ be an open and bounded set. A point $z_0 \in \partial \Omega$ is regular for  \eqref{diri-parab} if and only if
	\begin{equation}\label{series-Wiener}
		\sum_{k=1}^{\infty}2^{\frac{kd}{2}}\textnormal{cap}^{\mathcal{T}}\left(\Omega^{c} \cap \big(\overline{\Omega_{2^{-k}}(z_0)}\backslash \Omega_{2^{-(k+1)}}(z_0)\big) \right) = +\infty,
	\end{equation}
where $\Omega_{r}(z_0) = \{(y,s)\in \mathbb{R}^{d}\times \mathbb{R}_{+}:\Gamma(x_0-y,t_0-s)>(4\pi r)^{-d/2}\}$ is the so-called heat ball and $\textnormal{cap}^{\mathcal{T}}(K)$ denotes thermal capacity over compact sets $K\subset \mathbb{R}^{d+1}$
\begin{equation}
	\textnormal{cap}^{\mathcal{T}}(K):=\sup\{\mu(K):\mu \in \mathfrak{M}^{+}(K) \text{ and }\Gamma\ast \mu\leq 1\text{ on }\mathbb{R}^{d+1}\},\nonumber
\end{equation}
and $\Gamma (x,t)$ is the fundamental solution of the heat equation.
\end{theorem}

Theorem \ref{ev} has been generalized to parabolic operators in divergence form with smooth variable coefficients by \cite[Theorem 1.1]{Garofolo1}, and to those with 
$C^1-$Dini continuous coefficients by \cite[Theorem 1.1]{FGL}. Additionally, the theorem has been extended to the degenerate parabolic operator $\mathscr{L}_a=\partial_t(\vert y\vert^a \cdot)-\text{div}(\vert y\vert^a\nabla\cdot)$ in  \cite{Hu} for $a\in(-1,1)$. 

Let us recall that $E\subset\mathbb{R}^{d+1}$ is called \emph{heat thin} at $z_0\in\mathbb{R}^{d+1}$ (see \cite{Watson}) if either $z_0\notin\overline{E}$, or $z_0\in\overline{E}$ and there exists a supertemperature $u$  defined in a neighborhood of $z_0$ such that
\[
u(z_0)<\liminf_{\substack{z\to z_0\\z\in E\setminus\{z_0\}}}u(z).
\]
As is well known, we can define heat thinness based on the balayage function $V_K=\widehat{R}^K_1$ of compact sets $K$. More precisely, $E\subset\mathbb{R}^{d+1}$ is heat thin at $z_0$ if there exists a time-backward parabolic ball $Q_r(z_0)=B_r(x_0)\times(t_0-r^2,t_0)$ such that $V_{E\cap \overline{Q_r(z_0)}}(z_0)<1$. Consequently, by the well-established relationship between boundary regularity and thinness, we obtain the following equivalences (see \cite[Lemma 1.3]{Lanconelli} and Theorem \ref{ev}) with $E=\Omega^c$ (a closed set):
\begin{equation}\label{E-G}
	E \text{ is heat } (1,2)\text{-thin at } z_0 \iff \int_0^1 \frac{\textnormal{cap}^{\mathcal{T}}(E\cap \Omega_r(z_0))}{r^{d/2}} \, \frac{dr}{r} < \infty.
\end{equation}
This classical characterization of heat thinness has been extended to arbitrary sets $E$ in \cite[Theorem 1.11]{Brzezina}.

This paper addresses two fundamental analytical questions arising from previous discussions, now considered in the context of parabolic potentials studied by Jones and Sampson \cite{Jones, Sampson}, who introduced the Bessel kernel $\mathscr{G}_\alpha(z)$ and the Riesz kernel $\Gamma^{\alpha}(z)$ in the parabolic setting (see Section \ref{sec-potential-measure}). Let $C_{\alpha,2}(\cdot)$ denote the parabolic Bessel capacity for $0<\alpha< n/2$ (see Section \ref{parabolic cap}), where $n=d+2$ is the homogeneous dimension of space-time $\mathbb{R}^{d+1} = \mathbb{R}^d_x \times \mathbb{R}_t$ endowed with the pseudo-distance $d_{\mathcal{P}}(z,z_0) = \max\{\vert x-x_0\vert, \vert t-t_0\vert^{1/2}\}$. For an arbitrary set $E \subset \mathbb{R}^{d+1}$, we investigate whether the following equivalences hold:
\begin{itemize}
	\item[(P$_1$)] For $z_0 \in \overline{E}$, we have the capacity condition
	\begin{equation}\label{capacity-integral}
		\int_{0}^{1} \frac{\dot C_{\alpha,2}(E \cap \varTheta^{2\alpha}_{r}(z_0))}{r^{\frac{n - 2\alpha}{2}}} \frac{dr}{r} < \infty,
	\end{equation}
where $\varTheta^{s}_{\rho}(z_0)$ denotes the fractional heat ball 
	\begin{equation}\label{heat-ball}
		\varTheta^{s}_{\rho}(z_0) = \left\{(y,\tau) : \Gamma^{s}(x_0 - y, t_0 - \tau) > c_{s} \rho^{-\frac{n- s}{2}}\right\},\quad 0<s<n \nonumber .
	\end{equation}

	\item[(P$_2$)] There exists a backward parabolic ball $\overline{Q_r(z_0)}$ and a $\dot C_{\alpha,2}-$capacitary measure $\mu^A$ supported on  $A=E\cap \overline{Q_r(z_0)}$  such that $(\Gamma^{2\alpha}\ast \mu^A)(z_0)<1 $. 
	\item[(P$_3$)] For $z_0 \in \overline{E}$, there exists a positive Radon measure $\mu \in \mathfrak{M}^{+}(\mathbb{R}^{d+1})$ such that  
	\begin{equation}\label{thin-para-bessel-linear}
		(\Gamma^{2\alpha} \ast \mu)(z_{0}) < \liminf_{\substack{z \to z_0 \\ z \in E \setminus \{z_0\}}} (\Gamma^{2\alpha} \ast \mu)(z).
\end{equation}
	
\end{itemize} 

The next task is to explore the parabolic $(\alpha,q)$-thinness for $q\neq 2$, based in time-backward parabolic balls. As noted in \cite{Wolff} (see also \cite{Hedberg}), the nonlinear elliptic Havin-Mazya potential $\mathcal{V}_{\alpha,q}^{\mu} = G_\alpha \ast (G_\alpha \ast \mu)^{q'-1}$ (\cite{HM}) does not suffice to characterize elliptic $(\alpha,q)$-thinness for all $1<q\le d/\alpha$. Thus, a natural definition of parabolic $(\alpha,q)$-thinness is proposed as follows:
\begin{definition}[Nonlinear parabolic thin sets] \label{parab-thin1}
	Let $E\subset \mathbb{R}^{d+1}$ and $1<q\le n/\alpha$ with $0<\alpha<n$. The set $E$ is parabolic $(\alpha,q)-$thin at $z_0 \in \mathbb{R}^{d+1}$ if $z_0 \notin \overline{E}$, or if $z_0 \in \overline{E}$ and
	\begin{equation}\label{thin-ball1}
		\Phi_{\textnormal{cyl}}(z_0)=\int_{0}^{1} \left( \frac{C_{\alpha,q}(E \cap Q_{r}(z_0))}{r^{n-\alpha q}} \right)^{q'-1} \frac{dr}{r} < \infty,
	\end{equation}
	where $C_{\alpha,q}$ denotes the parabolic Bessel capacity (see Section \ref{parabolic cap}) and $q'=q/(q-1)$.
	\end{definition}

	This definition implies the one based on fractional heat balls, where \eqref{thin-ball1} is replaced by
	\begin{equation}\label{thin-heat-ball1}
	\Phi_{\text{heat}}^{s}(z_0)=\int_{0}^{1} \left( \frac{\dot C_{\alpha,q}(E \cap \varTheta^{s}_r(z_0))}{r^{\frac{n-s}{2}}} \right)^{q'-1} \frac{dr}{r} < \infty, \quad s=q\alpha<n,
	\end{equation}
	since $\varTheta^{s}_r(z_0) \subset Q_{c\sqrt{r}}(z_0)$ for some $c>0$, see \eqref{heatBall-into-parabBall}. In particular, if $E$ is parabolic $(1,2)-$thin at $z_0$ in the sense of Definition \ref{parab-thin1}, then
\begin{equation}
	\int_{0}^{1} \frac{\text{cap}^{\mathcal{T}}(E \cap \Omega_r(z_0))}{r^{d/2}} \frac{dr}{r}\lesssim 
	c^{d}\int_{0}^{1} \frac{C_{1,2}(E \cap Q_{s}(z_0))}{s^{d}} \frac{ds}{s} < \infty,\nonumber
\end{equation}
because $C_{1,2}(Q_r) \simeq \dot{C}_{1,2}(Q_r) = \text{cap}^{\mathcal{T}}(Q_r)$ for $0<r<1$ (see Proposition \ref{basic-properties-cap} and Corollary \ref{Cap-and-Cap_T}). The converse, however, appears unlikely in view of the geometric non-inclusion $Q_1(0)\nsubseteq \varTheta_1^{\alpha}(0)$, see Section \ref{fract-heat-balls}. This geometric discrepancy between the two notions of thinness does not affect our analysis, which is formulated entirely in terms of backward parabolic balls and the associated parabolic dyadic lattice $\mathscr{D}$, defined backward in time and adapted to the scaling $\delta_c(x,t) = (c x, c^2 t)$.

One of the main results of this paper is the proof that Definition \ref{parab-thin1} is equivalent to a modified version of the inequality \eqref{thin-para-bessel-linear}, where $\Gamma^{2\alpha}$ is replaced by nonlinear parabolic Wolff's potential. Precisely, we have: 
\begin{theorem}\label{Conjuntos finos Geral parabolico1}
	Let $\alpha>0$ be such that $1<q \leq n/\alpha$. The set $E \subset \mathbb{R}^{d+1}$ is parabolic $(\alpha,q)-$thin at $ z_0 \in \overline{E}$ if and only if there exists $\mu \in \mathfrak{M}^{+}(\mathbb{R}^{d+1})$ such that
	\begin{equation}
		({\mathbf{W}}_{\alpha,q}\mu) (z_0)< \liminf_{\substack{z \to z_0 \\ z \in E \setminus \{z_0\}}}({\mathbf{W}}_{\alpha,q}\mu)(z),
	\end{equation}
	where $\mathbf{W}_{\alpha,q}\mu$ denotes the  parabolic Wolff potential 
	\begin{equation}\label{Wolff-intro}
		(\mathbf{W}_{\alpha,q}\mu)(z) = \int_{0}^{1} \left( \frac{\mu(B_r(x)\times (t-r^2,t))}{r^{n-\alpha q}} \right)^{q'-1} \frac{dr}{r}.
	\end{equation}
\end{theorem} 
In order to prove this theorem (see Section \ref{non-parab-thin}), our main challenge was the introduction of a parabolic dyadic grid $\mathscr{D}$ in $\mathbb{R}^{d+1}_+$, in the sense of Lerner and Nazarov \cite{Lerner}, that is intrinsically connected to the time-backward parabolic rectangles $R^{-}$ (see Figure \ref{fig:exemplo1}) and the parabolic scaling $\delta_c(x,t) = (c x, c^2 t)$. This grid needed to be designed in such a way that it covers the entire space $\mathbb{R}^{d+1}_+$, not just a portion given by the parabolic dyadic grid $\mathfrak{D}(R^{-})$ as studied by Saari \cite{Saari} (see details in Section \ref{dyadic-wolff}). The parabolic dyadic lattice (see Definition \ref{par-dyad-lattice}) played a key role in the construction of a dyadic analogue of Wolff's potential within the framework of intrinsic parabolic geometry. As a consequence, we establish the parabolic dyadic Wolff's inequality (see Theorem \ref{gen-Wolff-ineq}), from which we also derive, as a byproduct:
\begin{theorem}[Parabolic Wolff's inequality]\label{cont-wolff-neq-int}
	Let $\mu \in \mathfrak{M}^+(\mathbb{R}^{d+1})$ and $0<\alpha < n$. If $1<q\le n/\alpha$, then  
	\begin{equation}
	\textnormal{E}_{\alpha,q}\mu :=\int_{\mathbb{R}^{d+1}} \vert \check{\mathscr{G}}_{\alpha}\mu\vert^{q'}dxdt\simeq \int_{\mathbb{R}^{d+1}} ({\textnormal{\textbf{W}}}_{\alpha,q}\mu)d\mu,\nonumber
	\end{equation}
	where $\check{\mathscr{G}}_{\alpha}\mu$ denotes the backward parabolic Bessel potential $\check{\mathscr{G}}_{\alpha}\mu(z)=\displaystyle \int_{\mathbb{R}^{d+1}} \mathscr{G}_{\alpha}(z-z_0)d\mu(z)$.
	\end{theorem}
	
	We also provide a homogeneous version of the previous theorem (see Theorem \ref{cont-Wolff-ineq}), which reads 
	\begin{equation}
	\dot{\textnormal{E}}_{\alpha,q}\mu \simeq \int_{\mathbb{R}^{d+1}} (\dot{\mathbf{W}}_{\alpha,q}\mu)d\mu,\nonumber
	\end{equation} 
	for all $0 < \alpha <n$ and $1<q<n/\alpha$, where 
	\begin{equation}\label{Wolff-intro2}
	(\dot{\mathbf{W}}_{\alpha,q}\mu)(z) = \int_{0}^{\infty} \left( \frac{\mu(B_r(x)\times (t-r^2,t))}{r^{n-\alpha q}} \right)^{{q'}-1} \frac{dr}{r}. 
	\end{equation}
	
The parabolic Wolff potentials, introduced by Kuusi and Mingione \cite{Mingione} (see also \cite{Mingione2}), are linked to the intrinsic geometry (see \cite{DiBenedetto}) of a degenerate evolution equation of $p$-Laplacian type, such as the equation $u_t - \text{div}(\vert D u\vert^{p-2}Du) = 0$. These equations resemble the heat equation on the scaled time-backward parabolic ball  $Q_{r}^{\lambda}(z) = B_r(x) \times (t - \lambda^{2-p}r^2, t)$, where $\lambda > 0$, as observed in \cite{Mingione}.

The parabolic dyadic Wolff inequality (Theorem~\ref{gen-Wolff-ineq}) and its continuous counterpart (Theorem~\ref{cont-wolff-neq-int}) provide the principal tools for the introduction of the parabolic dyadic Bessel capacity
\begin{equation}
\mathscr{C}_{\alpha,q}(K)^{1/q} := \sup\left\{ \Vert\mu\Vert \,:\, \mu \in \mathfrak{M}^{+}(K) \text{ and } {\mathcal{E}}\mu \le 1 \right\},\nonumber
\end{equation}
where $K \subset \mathbb{R}^{d+1}$ is a compact set and $\mathcal{E}$ denotes the regularized dyadic nonlinear energy \eqref{d-reg-energy}. Moreover, $\mathscr{C}_{\alpha,q}(K)$ is equivalent to the standard parabolic Bessel capacity $C_{\alpha,q}(K)$, in view of Lemma~\ref{key-equi-energias} and the dual formulation of capacity given in Proposition~\ref{duality}. Moreover, by the concept of $\mathscr{C}_{\alpha,q}-\,$extremal measure (see Definition \ref{extremal-meas}) and the equilibrium ($\mathscr{C}_{\alpha,q}-\,$capacitary) measure (see Definition \ref{capacitary-meas-Wolff}) for compact sets, we provide several characterizations of the parabolic dyadic Bessel capacity (see Theorems \ref{Thm-carac-capacity} and \ref{Cor-carac-capacity}) based on the parabolic regularized dyadic Wolff potential 
\begin{equation}
{\mathcal{W}}^{\mathcal{D}}_{\alpha,q}\mu = \sum_{\substack{R \in \mathscr{D} \\ \ell_R < 1}} \left( \ell_R^{\alpha q - n} \mu(\eta_R) \right)^{q'-1} \, \eta_R, 
\end{equation}
where $\eta_R\in C_c^{\infty}(\mathbb{R}^{d+1})$ and  $\mu(\eta_R):=\int_{\mathbb{R}^{d+1}} \eta_R d\mu$, see Section \ref{parab-dyadic-capacity} for a careful analysis. Therefore, with all the aforementioned results at our disposal, let 
\begin{equation}\label{thin-set}
e_{\alpha,q}^{\mathscr{P}}(E) = \left\{ z_0 \in \mathbb{R}^{d+1} \,:\,
\int_{0}^{1} \left( \frac{C_{\alpha,q}(E \cap Q_{r}(z_0))}{r^{n-\alpha q}} \right)^{q'-1} \frac{dr}{r} < \infty \right\},
\end{equation}
for all $\alpha>0$ such that $1<q\le n/\alpha$. Then we established the following parabolic Kellogg property (see Proposition \ref{Kellogg}):  
\begin{equation}
C_{\alpha,q} \big( E \cap e_{\alpha,q}^{\mathscr{P}}(E) \big) = 0. \nonumber
\end{equation}
The Kellogg property is also a consequence of the Choquet property, which is central to the analysis of fine continuity.

\begin{theorem}[Parabolic Choquet property]
	Let $\alpha>0$  be such that  $1<q\leq n/\alpha$. For arbitrary set $E \subset \mathbb{R}^{d+1}$ and $\varepsilon>0$,  there exists an open set $G \subset \mathbb{R}^{d+1}$ such that 
	\begin{equation}
		e^{\mathscr{P}}_{\alpha,q}(E) \subset G \;\;\text{ and  }\;\; \mathscr{C}_{\alpha,q}(E \cap G)< \varepsilon.
	\end{equation}
\end{theorem}

In addition, we establish the corresponding Kellogg property formulated in terms of fractional heat balls (see Proposition \ref{Kellogg-heat}) 
\begin{equation}\label{KP2}
C_{\alpha,2} \big( E \cap e_{\alpha,2}^{\text{heat}}(E) \big) = 0, 
\end{equation}
where $e_{\alpha,2}^{\text{heat}}(E)=\{z_0\in \mathbb{R}^{d+1}\,\colon\,\Phi_{\text{heat}}^{2\alpha}(z_0)<\infty\}$. Since the degenerate parabolic operator $\mathscr{L}_a$ is an extension operator of $(\partial_t - \Delta)^{\alpha/2}$ for $0<\alpha<2$, the Kellogg property \eqref{KP2} implies that the Wiener test of \cite[Theorema 1.1]{Hu} characterizes $\mathscr{L}_a-$regular boundary points $C_{\alpha,2}-\,$almost everywhere. In particular, the set of points $z_0 \in \partial\Omega$ at which the Wiener series in Theorem \ref{ev} converges is $\text{cap}^{\mathcal{T}}-\,$negligible. 

To conclude the introduction, we note that Lemma \ref{lemao} and Theorem \ref{Conjuntos finos Geral parabolico1} together imply the following remark.

\begin{remark} \label{remark-intro}Let $0<\alpha<n$ and $1<q\leq n/\alpha$. The statements are equivalent:
	\begin{itemize}
		\item[(i)] $E$ is parabolic $(\alpha,q)-$thin at $z_0$.
		\item[(ii)] Let $V$ be a neighborhood of $z_0$. If $V$ is sufficiently small, there exists an equilibrium  measure $\mu^{A}$ for $A=E\cap V$ such that 
	   \begin{equation}
		({\mathbf{W}}_{\alpha,q} \mu^{A}) (z_0)<1.
		\end{equation}
		\item[(iii)] For $z_0\in \overline{E}$, there exists $\mu \in \mathfrak{M}^{+}(\mathbb{R}^{d+1})$ such that
		\begin{equation}
			({\mathbf{W}}_{\alpha,q} \mu) (z_0)< \liminf_{\substack{z \to z_0 \\ z \in E \setminus \{z_0\}}}({\mathbf{W}}_{\alpha,q} \mu)(z).\nonumber
		\end{equation}
	\end{itemize}
\end{remark}
The preceding remark constitutes a nonlinear analogue of \eqref{E-G}, replacing the heat-ball formulation by backward parabolic ball framework. For $(\alpha,2)-$thinness defined via fractional heat balls, Lemma~\ref{olema-heat} together with the second part of the proof of Theorem~\ref{Conjuntos finos Geral parabolico1}, with $\Gamma^{2\alpha}\ast\mu$ in place of $\mathbf{W}_{\alpha,q}\mu$, yields the implications
\[
P_1\Rightarrow P_2\Rightarrow P_3,
\]
for all $0<\alpha<n/2$. The implication $P_3\Rightarrow P_1$ remains open for $\alpha\in (0,n/2)$. 

\section{Preliminaries}

In this paper the space-time $\mathbb{R}^{d+1}=\mathbb{R}^d_x\times \mathbb{R}_t=\{(x,t)\,:\, x\in \mathbb{R}^d,\, t\in\mathbb{R}\}$, $d\geq 2$, is endowed by anisotropic pseudo-distance 
\begin{equation}
    d_{\mathcal{P}}((x,t), (y,s))=\max\left\{\vert x-y\vert, \vert t-s\vert^{1/2}\right\}
\end{equation}
which is equivalent to the unique solution $\rho(z)$ (parabolic distance) of the equation (see  \cite{FR} for details)
\begin{equation}\label{eq-norm}
	\frac{\vert x\vert^2}{\rho^2} +\frac{\vert t\vert^2}{\rho^4}=1.
\end{equation}
The pseudo-distance $d_{\mathcal{P}}$ is homogeneous of degree one with respect to dilation 
\begin{equation}\label{dil-parab}
	(x,t)\mapsto \delta_{\gamma}(x,t):=(\gamma x,\gamma^2 t), \quad \gamma>0.
\end{equation}
The metric measure space $\left(\mathbb{R}^{d+1}, d_{\mathcal{P}}, d\mathscr{L}\right)$ with  Lebesgue measure $d\mathscr{L}=dxdt$ is a space of homogeneous type, in the sense of Coifman and Weiss \cite[Chapter 3]{CW}, where its homogeneous dimension is given by $n:=\log_2c_{D}=d+2$. Here,  $c_{D}=2^{d+2}$ is the constant of the doubling measure 
\begin{equation}
\mathscr{L}(\delta_2C_r)\leq 2^{d+2}\mathscr{L}(C_r)\nonumber
\end{equation}
for every full parabolic cylinder $C_r(x,t)=B_r(x)\times (t-r^2, t+r^2)$ in the space-time $\mathbb{R}^d_x\times \mathbb{R}_t$. The scaling $(x,t)\mapsto \delta_{\gamma}(x,t)$ naturally motivated the introduction of the following change of coordinates (see Fabes and Riviere \cite{FR}) 
\begin{align}\label{polar-coord}
	\begin{cases}
	x_d&=\quad\rho \sin\varphi_1\cdots \sin\varphi_{d-1}\sin \varphi_d\\
	x_{d-1}&=\quad\rho \sin\varphi_1\cdots \sin\varphi_{d-1}\cos \varphi_d\\
	&\vdots  \\
	x_{1}&=\quad\rho \sin\varphi_1\cos \varphi_2\\
	t&=\quad\rho^2\cos\varphi_1,
	\end{cases}
\end{align}
where $0\leq \varphi_d\leq 2\pi$ and $0\leq \varphi_j\leq \pi$,  for all $j=1,\cdots, d-1$. Hence, the element of volume  is given by 
\begin{equation}
	d\mathcal{L}=\rho^{n-1}(1+\cos^2\varphi_1)d\rho d\sigma,
\end{equation}
where $d\sigma$ denotes the surface measure over the parabolic unit sphere $\Sigma_d=\big\{(x,t)\in\mathbb{R}^{d}\times\mathbb R_+\,:\,\frac{\vert x\vert^2}{\rho^2} +\frac{t^2}{\rho^4}=1\big\}$.

\subsection{Parabolic potential of measures}\label{sec-potential-measure}

Let us define the function 
\begin{equation}\label{Riesz-kernel}
	 {\Gamma}^{\alpha}(x,t) =
		[(4\pi)^{\frac{d}{2}}\Gamma({\alpha}/{2})]^{-1}t^{-\frac{n-\alpha}{2}}e^{-\frac{\vert x\vert ^{2}}{4t}}\mathds{1}_{\{t>0\}}, \quad \alpha>0
   \end{equation}
which is called parabolic Riesz kernel. 
The kernel ${\Gamma}^{\alpha}(x,t)$ is homogeneous of degree $\alpha-n$ with respect to the parabolic scaling \eqref{dil-parab}, that is, 
\begin{equation}\label{scaling-Gamma}
	{\Gamma}^{\alpha}(\lambda x,\lambda^{2}t) = \lambda^{\alpha - n}{\Gamma}^{\alpha}(x,t).
\end{equation}
Furthermore, a straightforward change of coordinates (see \eqref{polar-coord}) shows that the kernel ${\Gamma}^{\alpha}(x,t)$ is locally integrable. Hence, its Fourier transform in $\mathcal{S}'(\mathbb{R}^{d+1})$ is well defined, and given by  (see \cite[Theorem 2.2]{Sampson} and \cite{Segovia-Wheeden}) 
\begin{equation}\label{syb-Gamma1}
	\widehat{{\Gamma}^{\alpha}}(\xi, \tau) = (\vert \xi\vert^{2}+i\tau)^{-\frac{\alpha}{2}} \quad \text{ for }\quad 0< \alpha< n.
\end{equation} 
Let $f\in\mathcal{S}(\mathbb{R}^{d+1})$ and $0<\alpha<n$.  The fractional integral of order $\alpha$, denoted by ${\Gamma}^{\alpha}f$, is defined by 
\begin{equation}
	({\Gamma}^{\alpha}f)(x,t)=\int_{\mathbb{R}^{d+1}}{\Gamma}^{\alpha}(x-y,t-s)f(y,s)dyds\nonumber
\end{equation}
and is called as parabolic Riesz potential of $f$.  Let $\mathfrak{M}(\mathbb{R}^{d+1})$ denote the space of Radon measures on $\mathbb{R}^{d+1}$ equipped with the weak$^{\star}-$ topology. For $\mu\in \mathfrak{M}(\mathbb{R}^{d+1})$, the parabolic Riesz potential of $\mu$ is defined by 
\begin{equation} \label{Riesz-potential}
	({\Gamma}^{\alpha}\mu)(x,t)=\int_{\mathbb{R}^{d+1}}{\Gamma}^{\alpha}(x-y,t-s)d\mu(y,s).
\end{equation}
Consider the parabolic Bessel kernel 
\begin{equation}
\mathscr{G}_{\alpha}(x,t) = (4\pi)^{-\frac{d}{2}}\Gamma({\alpha}/{2})^{-1} t^{-\frac{n-\alpha}{2}}\exp\left\{-t-\frac{\vert x\vert^{2}}{4t}\right\}\mathds{1}_{\{t>0\}},\quad \text{Re}\,\alpha >0.
\end{equation}
This kernel $\mathscr{G}^{\alpha}\in L^1(\mathbb{R}^{d+1})$ with $\Vert \mathscr{G}^{\alpha}\Vert_{L^1}=1$ and its Fourier transform (see \cite{Jones}) is given by 
\begin{align}\label{syb-Gamma2}
	\widehat{\mathscr{G}_{\alpha}}(\xi,\tau) = (1+\vert \xi\vert^{2}+i\tau)^{-\frac{\alpha}{2}} \quad \text{as}\quad \alpha>0.
\end{align}
By Young's inequality in $L^p$,  the Bessel potential 
\begin{equation}
	\mathscr{G}_{\alpha} f (x,t)=\int_{\mathbb{R}^{d+1}} \mathscr{G}_{\alpha}(x-y,t-s)f(y,s)dyds\nonumber
\end{equation}
is well-defined, and $\mathscr{G}_{\alpha}$ extends to a bounded operator on $L^p(\mathbb{R}^{d+1})$ for all $1\leq p\leq \infty$. It is well-known that the kernels ${\Gamma}^{\alpha}(x,t)$ and $\mathscr{G}_{\alpha}(x,t)$ are lower semi-continuous, and the change of coordinates \eqref{polar-coord} readily yields the following properties:
\begin{itemize}
	\item[\textit{(a)}] ${\Gamma}^{\alpha}\notin L^{q'}(\mathbb{R}^{d+1})$, for all  $0<\alpha<n$.
	\item[\textit{(b)}] $\displaystyle \int_{\rho>1}[{\Gamma}^{\alpha}(x,t)]^{q'}d\mathcal{L}<\infty$, only if $\alpha q<n$.
	\item[\textit{(c)}] $\mathscr{G}_{\alpha}\notin L^{q'}(\mathbb{R}^{d+1})$, only if $\alpha q\leq n$.
	\item[\textit{(d)}] $\displaystyle \int_{\rho>1}[\mathscr{G}_{\alpha}(x,t)]^{q'}d\mathcal{L}<\infty$, for all $\alpha>0$,
\end{itemize}
where $1<q<\infty$. The previous properties are similar to the elliptic Bessel and Riesz kernels, however the special role in the time variable brings new challenges. For instance,  this affects the analysis of the dyadic and continuous Wolff's inequality associated to these kernels (see Section \ref{Wolff-inequality}). 

\subsection{Fractional heat balls}\label{fract-heat-balls}
Let $z=(x,t)\in\mathbb{R}^{d+1}$ and $\rho>0$, we define the \textit{fractional heat ball} $\varTheta^{\alpha}_{\rho}(z)$ centered at $z$ with radius $\rho$  by level set  
\begin{equation}\label{bola do calor fracionaria}
	\varTheta^{\alpha}_{\rho}(z) = \left\{(y,s)\,:\,{\Gamma}^{\alpha}(x-y,t-s)>c_{\alpha}\rho^{-\frac{n-\alpha}{2}}\right\},\quad c_{\alpha}=(4\pi)^{-{d}/{2}}\Gamma({\alpha}/{2})^{-1},
	\end{equation}
where $0<\alpha<n$.   A short computation shows that 
\begin{equation}
	\varTheta^{\alpha}_{\rho}(z) = \left\{(y,s)\,:\,  t-\rho<s<t,
\,\;\; \vert x-y\vert<r_{\rho}(s)  \right\}\nonumber,
\end{equation}
where $r_{\rho}(s)=\left[2(n-\alpha)(t-s)\log \left(\frac{\rho}{t-s}\right)\right]^{1/2}$. Then $r_{\rho}(s)\leq c\sqrt{\rho}\,$ with  $c=\sqrt{2(n-\alpha)/e}$, which yields  
\begin{equation}\label{heatBall-into-parabBall}
	\varTheta_{\rho}^{\alpha}(z)\subset  Q_{c\sqrt{\rho}}(z).
\end{equation}

Note that the center $z$ of the fractional heat ball lies on its boundary, which decomposes as  
\begin{equation}
	\partial \varTheta^{\alpha}_{\rho}(z) = \Gamma_{\text{lat}}\cup \Gamma_{\text{pas}}\cup \Gamma_{\text{pre}},
\end{equation}
where the lateral boundary $\Gamma_{\text{lat}}=\{(y,s)\colon t-\rho<s<t, \vert y-x\vert=r_{\rho}(s)\}$ is a smooth surface, the  furthest point in the past-time corresponds to $\Gamma_{\text{pas}}
=\{(x, t-\rho)\}$, and the present-time point $\Gamma_{\text{pre}}
=\{z\}$ marks the limiting position as $s\rightarrow t^{-}$ (see Figure \ref{fig:boundary}). 
\begin{figure}[htbp]
	\centering 
	\includegraphics[width=0.6\textwidth]{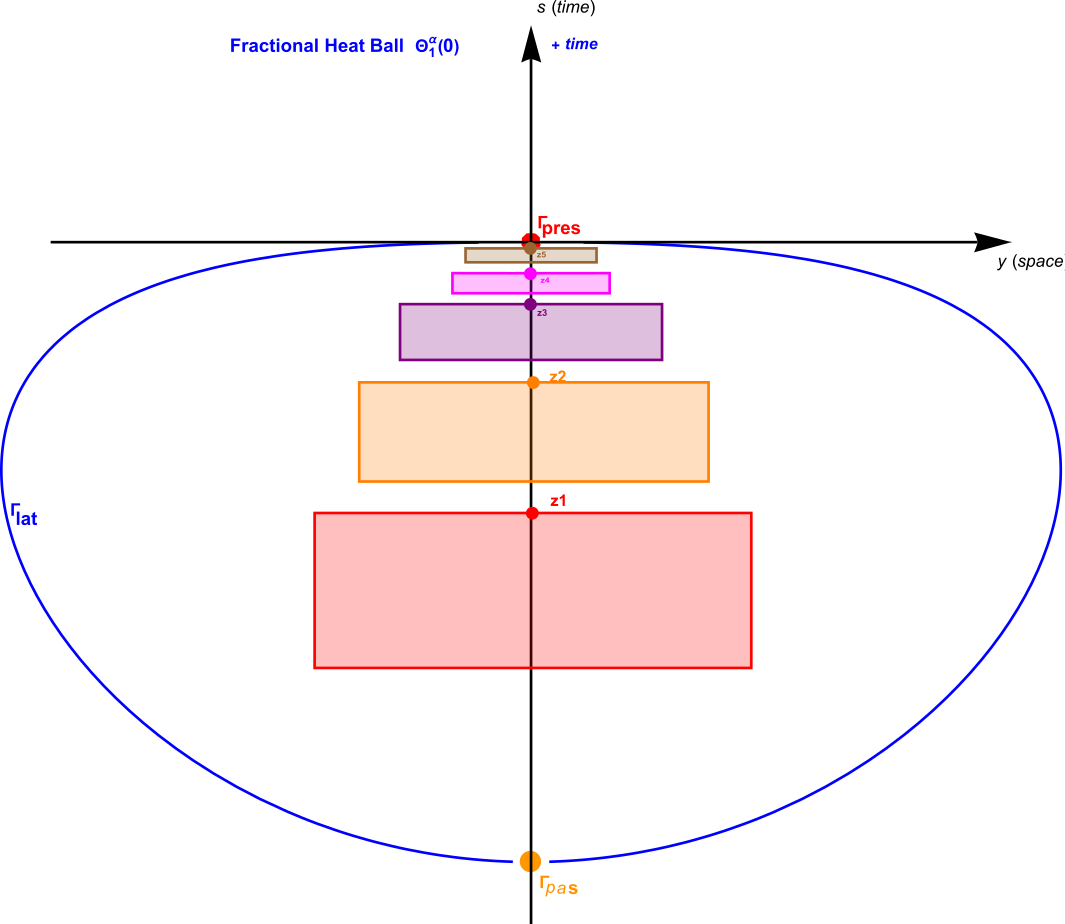}
\caption{Fractional heat ball centered at $z=0$ (red) and time-backward balls $Q_{c}(z_k)$}with center $z_{k}=(0,t_k)$.
\label{fig:boundary}
\end{figure}
Consequently, $Q_{c\sqrt{\rho}}(z) \nsubseteq  \varTheta_{\rho}^{\alpha}(z)$ for every $c > 0$, since $z\in \partial \varTheta^{\alpha}_{\rho}(z)$. More precisely, by scaling and translation it suffices to consider the case $z=0=(0,0)$ and $\rho=1$. Define 
\[
z_k=(x_k,t_k)
=
\left(
\frac{2}{k}\left(1+\frac{1}{k}\right)\sqrt{(n-\alpha)\log k}\,e,
-\frac{1}{k^2}
\right), \quad k\in\mathbb{N}
\]
for every $e\in\mathbb{R}^d$ with $|e|=1$.
Then
\[
z_k\in Q_1(0)=B_1(0)\times(-1,0),
\]
while
\[
\Gamma^\alpha(-x_k,-t_k)
=
c_\alpha k^{\,n-\alpha}
\exp\!\left(
-(n-\alpha)\left(1+\frac{1}{k}\right)^2\log k
\right)
=
c_\alpha
k^{(n-\alpha)\bigl[1-(1+\frac{1}{k})^2\bigr]}.
\]
It follows that
\[
\Gamma^\alpha(-x_k,-t_k)<c_\alpha
\]
for all sufficiently large $k$. Hence $z_k\notin\varTheta_1^{\alpha}(0)$.

In parallel with the elliptic Riesz potential, which can be represented via Euclidean balls (see \cite[p.~77]{AH}), the parabolic Riesz potential admits a formulation in terms of heat balls.
	
\begin{lemma}\label{Riesz-heat-ball} Let $0<\alpha <n$ and $\mu\in\mathfrak{M}^+(\mathbb{R}^{d+1})$. Then, up to constant depending only on $n$ and $\alpha$, one has 
	\begin{equation}
		({\Gamma}^{\alpha}\mu)(z) =  \int_{0}^{\infty}r^{-\frac{n-\alpha}{2}}\mu(\varTheta^{\alpha}_{r}(z))\frac{dr}{r}.\nonumber
	\end{equation}
\end{lemma}
\begin{proof} 
Since $\Gamma^{\alpha}$ is nonnegative and locally integrable in $\mathbb{R}^{d+1}$,  by Cavalieri's principle and  change of variable $r \rightarrow c_{\alpha}r^{-\frac{n-\alpha}{2}}$ we have 
	$$({\Gamma}^{\alpha}\mu)(z) = \int_{0}^{\infty}\mu(\{(y,s)\,:\,\Gamma^{\alpha}(x-y,t-s)>r\})dr\simeq \int_{0}^{\infty}r^{-\frac{n-\alpha}{2}}\mu(\varTheta^{\alpha}_{r}(z))\frac{dr}{r},$$
as desired. 
\end{proof}

To finish this section, consider the parabolic  backward Riesz potential $\check{{\Gamma}}^{\alpha}\mu$ defined by 
\begin{equation}
    (\check{{\Gamma}}^{\alpha}\mu)(y,s)=\int_{\mathbb{R}^{d+1}} {\Gamma}^{\alpha}(z-y,\tau-s)d\mu(z,\tau).\nonumber
\end{equation}

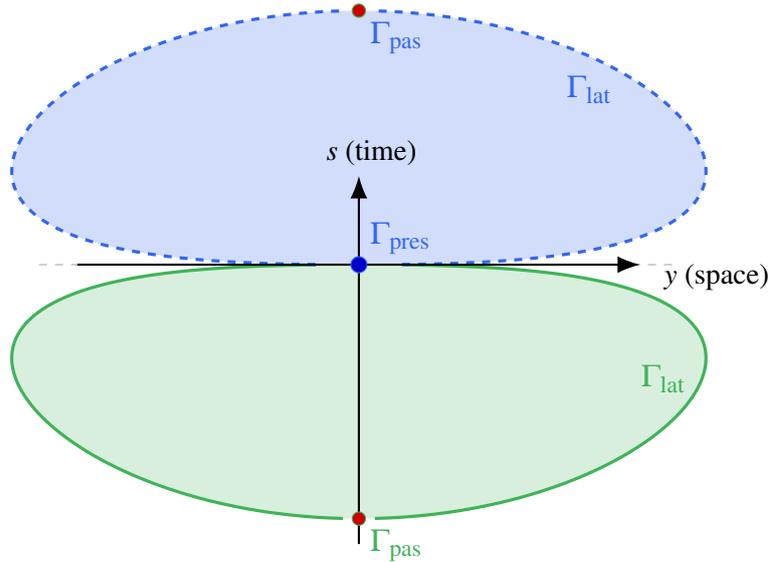
\begin{figure}[htb]
	\centering 

\begin{tikzpicture}
	\pgfmathsetmacro{\d}{2}
	\pgfmathsetmacro{\alpha}{1.5}
	\pgfmathsetmacro{\n}{\d + 2}
	\pgfmathsetmacro{\rho}{1.0}
	\def\eps{0.0008}
  
	\definecolor{backwardcol}{RGB}{64,179,89}   
	\definecolor{forwardcol}{RGB}{51,102,242}   
  
	\begin{axis}[
	  width=12cm, height=10cm,
	  xmin=-1.25, xmax=1.25, ymin=-1.25, ymax=1.25,
	  axis equal image,
	  axis line style={draw=none},
	  ticks=none,
	  clip=false
	]
  
	\addplot [
	  fill=backwardcol!45, fill opacity=0.45, draw=none,
	  domain=-\rho+\eps:-\eps, samples=600, variable=\s
	] 
	({ sqrt(2*(\n-\alpha)*(-\s)*ln(\rho/(-\s))) }, {\s})
	-- plot[domain=-\eps:-\rho+\eps, samples=600, variable=\s]
	   ({ -sqrt(2*(\n-\alpha)*(-\s)*ln(\rho/(-\s))) }, {\s})
	-- cycle;
  
	\addplot[very thick, backwardcol, domain=-\rho+\eps:-\eps, samples=600, variable=\s]
	  ({ sqrt(2*(\n-\alpha)*(-\s)*ln(\rho/(-\s))) }, {\s});
	\addplot[very thick, backwardcol, domain=-\rho+\eps:-\eps, samples=600, variable=\s]
	  ({ -sqrt(2*(\n-\alpha)*(-\s)*ln(\rho/(-\s))) }, {\s});
  
	\addplot [
	  fill=forwardcol!50, fill opacity=0.45, draw=none,
	  domain=\eps:\rho-\eps, samples=600, variable=\s
	] 
	({ sqrt(2*(\n-\alpha)*(\s)*ln(\rho/(\s))) }, {\s})
	-- plot[domain=\rho-\eps:\eps, samples=600, variable=\s]
	   ({ -sqrt(2*(\n-\alpha)*(\s)*ln(\rho/(\s))) }, {\s})
	-- cycle;
  
	\addplot[very thick, dashed, forwardcol, domain=\eps:\rho-\eps, samples=600, variable=\s]
	  ({ sqrt(2*(\n-\alpha)*(\s)*ln(\rho/(\s))) }, {\s});
	\addplot[very thick, dashed, forwardcol, domain=\eps:\rho-\eps, samples=600, variable=\s]
	  ({ -sqrt(2*(\n-\alpha)*(\s)*ln(\rho/(\s))) }, {\s});
  
	\addplot[gray!70, dashed, domain=-1.25:1.25] ({x},{0});
	\addplot[thick, -{Latex[length=3mm]}, black, domain=-1.1:1.1] ({x},{0});
	\addplot[thick, -{Latex[length=3mm]}, black, domain=-1.1:0.35] ({0},{x});
	\node[black, right] at (axis cs:1.15,-0.05) {\small $y$ (space)};
	\node[black, above] at (axis cs:0.05,0.35) {\small $s$ (time)};
  
	\addplot+[mark=*, mark size=2.5, only marks, backwardcol] coordinates {(0,-1)};
	\node[below right, backwardcol] at (axis cs:0,-1) {$\Gamma_{\mathrm{pas}}$};
	\node[backwardcol] at (axis cs:1.19,-0.45) {$\Gamma_{\mathrm{lat}}$};
  
	\addplot+[mark=*, mark size=3, only marks, forwardcol] coordinates {(0,0)};
	\node[above right, forwardcol] at (axis cs:0,0) {$\Gamma_{\mathrm{pres}}$};
	\node[forwardcol] at (axis cs:0.9,0.69) {$\Gamma_{\mathrm{lat}}$};
	\addplot+[mark=*, mark size=2.5, only marks, backwardcol] coordinates {(0,1)};
	\node[below right, forwardcol] at (axis cs:0,1) {$\Gamma_{\mathrm{pas}}$};
	\end{axis}
  \end{tikzpicture} 
\caption{The set $\check{\varTheta}_{r}^{\alpha}(0)$ (blue) is simply the reflected counterpart of the heat ball (green).}
\label{fig:exemplo2}
\end{figure}

The potential $\check{\Gamma}^{\alpha}\mu$ is fundamental in the construction of the dual capacity (see Section \ref{parabolic cap}). Moreover, as in Lemma \ref{Riesz-heat-ball}, the potential $\check{\Gamma}^{\alpha}\mu$ admits the following representation  
\begin{equation}
	(\check{\Gamma}^{\alpha}\mu) (y,s)=\int_0^{\infty}r^{-\frac{n-\alpha}{2}}\mu(\check{\varTheta}_{r}^{\alpha}(y,s))\frac{dr}{r},
\end{equation}
where $\check{\varTheta}_{r}^{\alpha}(y,s)=\{(x,t)\,:\,{\Gamma}^{\alpha}(x-y,\,t-s)>c_{\alpha}\rho^{-\frac{n-\alpha}{2}}\}$ which is a reflection on time of the heat ball (see Figure \ref{fig:exemplo2}).

\subsection{Parabolic capacity}\label{parabolic cap}

Let $1<q<\infty$ and $0<\alpha<n=d+2$. For an arbitrary set $A \subset \mathbb{R}^{d+1}$, the parabolic Riesz capacity $\dot{C}_{\alpha,q}(A)$ is defined by 
	\begin{equation}\label{capacidade parabolica de Riesz}
		\dot{C}_{\alpha,q}(A) = \inf\left \{\int\vert f\vert^{q}d\mathcal{L}\,:\, f\in\dot{\Omega}_A\right\},
	\end{equation}
where $\dot{\Omega}_{A}$ denotes the set of nonnegative functions $f\in L^{q}(\mathbb{R}^{d+1})$ such that ${\Gamma}^{\alpha}\ast f\geq \mathds{1}_{A}$. If $\dot{\Omega}_{A} = \emptyset$, we define $\dot{C}_{\alpha,q}(A) = +\infty$.  For $\alpha>0$, the parabolic Bessel capacity $C_{\alpha,q}$ of the set $A$ is defined as follows 
	\begin{equation}\label{capacidade parabolica de bessel}
		C_{\alpha,q}(A) = \inf\left \{\int\vert f\vert^{q}d\mathcal{L}\,:\,f\in \Omega_{A}\right\},
	\end{equation}
	where $\Omega_{A}= \{f\in L^{q}\,:\,f\geq 0\text{ and }\mathscr{G}_{\alpha}\ast f \geq \mathds{1}_{A}\}$. 
	
\begin{proposition}\label{basic-properties-cap}\ 
	\begin{itemize}
	\item[(i)] The set function  ${C}_{\alpha,q}$ (resp. $\dot{C}_{\alpha,q}$) is a capacity in the sense Meyers \cite{Meyers}. 
	\item[(ii)]The set function $C=\dot{C}_{\alpha,q}$ (resp. $C={C}_{\alpha,q}$) is an outer capacity 
	\begin{equation} 
	C(A)=\overline{C}(A):=\inf\{ C(E)\,:\, E\supset A\; \text { for open sets } E\}.
	\end{equation} 
	\item[(iii)] Let $\tau_{(x,t)}$ be a translation, then $\dot{C}_{\alpha,q}(\tau_{(x,t)}A)=\dot{C}_{\alpha,q}(A)\,$ for all $(x,t)\in\mathbb{R}^{d+1}$. Moreover, we have the scaling property 
	\begin{equation}\label{scaling-capacity}
        \dot{C}_{\alpha,q}(\delta_{\lambda}A)=\lambda^{n-\alpha} \dot{C}_{\alpha,q}(A), \quad \lambda>0.
    \end{equation}
	\item [(iv)] Let $0<\alpha<n$ such that $1<q<n/\alpha$. Then $\dot{C}_{\alpha,q}(\varTheta^{\alpha}_r)>0$ and 
     \begin{equation}\label{cap-Ball}
		\dot{C}_{\alpha,q}(\varTheta^{\alpha}_r)=r^{{(n-\alpha q)}/{2}} \,\dot{C}_{\alpha,q}(\varTheta^{\alpha}_1) \quad \text{ for }\;\; r>0.
	\end{equation}
    \item[(v)] Let $0<\alpha<n$ such that $1<q<n/\alpha$. Then for every compact set $K$ we have 
    \begin{equation}
        \dot{C}_{\alpha,q}(K)\leq {C}_{\alpha,q}(K)\lesssim \dot{C}_{\alpha,q}(K) + \dot{C}_{\alpha,q}(K)^{n/(n-\alpha q)}.
    \end{equation}
    In particular, one has $C_{\alpha,q}(B_r\times (t-r^2,t+r^2))\simeq r^{n-\alpha q}$ as  $\,0<r<1$.
    \item[(vi)] If $\alpha q=n$, then $C_{\alpha,q}(B_r\times (t-r^2,t+r^2))\simeq (\log (1/r))^{1-q}\,$ for $0<r<1/2$. 
	\end{itemize} 
\end{proposition}
\begin{proof} The proofs of statements $(i)$ and $(ii)$ follows by standard argument, see \cite{Meyers, AH}. To show $(iii)$ it is sufficient recall that 
	$\Vert \delta_\lambda f\Vert_{L^{p}(\mathbb{R}^{d+1})}=\lambda^{-{n}/{p}}\Vert f\Vert_{L^{p}(\mathbb{R}^{d+1})}$ and ${\Gamma}^{\alpha}(\delta_\lambda f)=\lambda^{-\alpha}\delta_{\lambda}({\Gamma}^{\alpha}f)$.  To show $(iv)$, recall that $\int_{\rho>1}[{\Gamma}^{\alpha}(x,t)]^{q'}dxdt<\infty$, if $1<q<n/\alpha$. Then employing the same reasoning as in  \cite[Proposition 2.6.1(c)]{AH}, we have $\dot{C}_{\alpha,q}(\varTheta_{1}^{\alpha})^{1/q} \geq c $. Now from $\delta_{\sqrt{r}}\,\varTheta^{\alpha}_{1}(x,t)=\varTheta^{\alpha}_{r}(x,t)$ and scaling \eqref{scaling-capacity} we obtain \eqref{cap-Ball}. The proof of $(v)$ and $(vi)$ can be found in \cite[Remark 4.14]{Nguyen}.
\end{proof}

According to Choquet's theorem (see \cite[Thm.\,2.3.11]{AH}), all Borel sets in $\mathbb{R}^{d+1}$ are $C_{{\alpha},q}-$capacitable and $\dot{C}_{{\alpha},q}-$capacitable, in view of  Proposition \ref{Prop-conv-capacity} bellow (see \cite[Thm.\,6 and Thm.\,7]{Meyers}).
	\begin{proposition}[\cite{Meyers}] \label{Prop-conv-capacity}Let $1<q<\infty$, $C=C_{{\alpha},q}$ as  $\alpha>0$ and  $C=\dot{C}_{{\alpha},q}$ as $0<\alpha<n$.  
		\begin{itemize}
		\item[(i)]If $K_1\supset K_2\supset \cdots $ are compact sets in $\mathbb{R}^{d+1}$ and $K=\cap_{j}K_j$, then 
	$
			\lim_{j\rightarrow \infty} C(K_j)=C(K).
	$
		\item[(ii)]If $A_1\subset A_2\subset \cdots $ are arbitrary sets in $\mathbb{R}^{d+1}$ and $A=\cup_{j}A_j$, then 
	$
			\lim_{j\rightarrow \infty} C(A_j)=C(A).
	$
	\end{itemize}
	\end{proposition}  
Since  $\dot{C}_{\alpha,q}$ satisfies (see \cite[Thm.\,2]{Meyers})
	\begin{equation}
		\dot{C}_{\alpha,q}(\{(x,t)\,:\, ({\Gamma}^{\alpha}\ast f)(x,t)\geq \lambda\})\leq \lambda^{-q}\int_{\mathbb{R}^{d+1}}\vert f\vert^q d\mathcal{L},
	\end{equation}
	then $\dot{C}_{\alpha,q}(A)=0$  if and only if $A\subset \{ (x,t)\,:\, ({\Gamma}^{\alpha}\ast f)(x,t)=+\infty\}$ for all $f\in L^{q}(\mathbb{R}^{d+1})$. Let $\mu\in\mathfrak{M}^+$ and $A\subset\mathbb{R}^{d+1}$ be a $\mu-$mensurable subset. The dual capacity of $\dot{C}_{\alpha,q}$ is defined by set function 
\begin{equation}
	\dot{\text{C}}\text{ap}_{\alpha,q}(A)=\sup\left\{\mu(A)\,:\, \mu\in\mathfrak{M}^+(A) \text{ and } \Vert \check{{\Gamma}}^{\alpha}\mu\Vert_{L^{q'}}\leq 1 \right\},
\end{equation} 
where $\mathfrak{M}^+(A)=\{\mu\in \mathfrak{M}^+(\mathbb{R}^{d+1})\,:\, \mu(A^c)=0\}$, for all $1<q<\infty$ and $0<\alpha<n$.
Similarly, the dual capacity of  $C_{\alpha,q}$ is defined by 
	\begin{equation}
		\ccap_{\alpha,q}(A)=\sup\left\{\mu(A)\,:\, \mu\in\mathfrak{M}^+(A) \text{ and } \Vert \check{\mathscr{G}}_{\alpha}\mu\Vert_{L^{q'}}\leq 1 \right\},
	\end{equation}   
for all $\alpha>0$. Mimicking \cite[Thm.\,3.5.1]{Aikawa}, we can show  that $\ccap_{\alpha,q}$ and $\CCap_{\alpha,q}$ are \textit{inner capacities} on the $\sigma-$field of the $\mu-$mensurable sets.  Moreover, invoking von Neumann's mini-max theorem and mimicking the proof of \cite[Thm.\,3.6.1]{Aikawa} (see also \cite{AH}), we get the following relation between parabolic capacity and its dual.

\begin{proposition}\label{duality}Let $1<q<\infty$ and $K\subset \mathbb{R}^{d+1}$ be compact set.  
	\begin{itemize}
		\item[(i)] If $0<\alpha<n$, then $[\dot{C}_{\alpha,q}(K)]^{1/q}={\CCap}_{\alpha,q}(K)$.
		\item[(ii)] If $\alpha>0$, then $[C_{\alpha,q}(K)]^{1/q}=\ccap_{\alpha,q}(K)$.
	\end{itemize}
\end{proposition}

Let $\mu\in\mathfrak{M}^{+}$ and $1<q<\infty$. The parabolic Havin-Mazya's potential associated to Riesz kernel ${\Gamma}^{\alpha}(x,t)$ is given by 
	\begin{equation}\label{homo-V}
		(\dHM\mu)(x,t)= \int_{\mathbb{R}^{d+1}} {\Gamma}^{\alpha}(x-y,t-s)[(\check{{\Gamma}}^{\alpha}\mu)(y,s)]^{{q'}-1}dy ds,
	\end{equation} 
	where  $0<\alpha<n$. By Fubini's theorem one has 
	\begin{equation}
		\int_{\mathbb{R}^{d+1}} (\dHM\mu)d\mu = \int_{\mathbb{R}^{d+1}} (\check{{\Gamma}}^{\alpha}\mu)^{q'}d\mathcal{L}. \label{Energy-Mazya}
	\end{equation}
Similarly, we have 
\begin{equation}
	\int_{\mathbb{R}^{d+1}} (\textnormal{\textbf{V}}_{\alpha,q}\mu)d\mu = \int_{\mathbb{R}^{d+1}} (\check{\mathscr{G}}_{\alpha}\mu)^{q'}d\mathcal{L}, \label{Energy-Mazya-Bessel}
\end{equation}
where  $\textnormal{\textbf{V}}_{\alpha,q}$ denotes the parabolic Havin-Mazya's potential associated to Bessel kernel 
	\begin{equation}
		(\textnormal{\textbf{V}}_{\alpha,q}\mu)(x,t)= \int_{\mathbb{R}^{d+1}} \mathscr{G}_{\alpha}(x-y,t-s)[(\check{\mathscr{G}}_{\alpha}\mu)(y,s)]^{q'-1}dy ds\nonumber, \quad \alpha>0.
	\end{equation} 

	\begin{proposition}[\cite{Aikawa}]\label{prop-aikawa}Let \(K \subset \mathbb{R}^{d+1}\) be compact with \(C_{\alpha,q}(K)<\infty\) for some \(\alpha>0\) and \(1<q<\infty\). Then there exists a measure \(\gamma^K\) supported on \(K\), i.e. \(\operatorname{spt}(\gamma^K)\subset K\), such that
		\begin{itemize}
			\item[(i)] $\gamma^{K}(K)={C}_{\alpha,q}(K)$
			\item[(ii)] $\textnormal{\textbf{V}}_{\alpha,q}\gamma^K= \mathscr{G}_{\alpha}\ast (\check{\mathscr{G}}_{\alpha}\gamma^K)^{q'-1} = \mathscr{G}_{\alpha} f^K \geq 1 \quad {C}_{\alpha,q}\text{-a.e}\; \text{ on }\,  K$
			\item[(iii)] $\textnormal{\textbf{V}}_{\alpha,q}\gamma^K\leq 1 \; \text{ on }\; \text{spt}(\gamma^K) $,
		\end{itemize}
		where the capacitary function $f^K=(\check{\mathscr{G}}_{\alpha}\gamma^K)^{q'-1}$ a.e\,  on \, $\mathbb{R}^{d+1}$ and   $\;({\mathscr{G}}_{\alpha}f^{K})\geq 1$\; ${C}_{\alpha,q}-$a.e\, on\,  $K$.  In particular, we have 
			\begin{equation}\label{sup1}
				{C}_{\alpha,q}(K)=\sup \left\{\mu(K)\,:\,\mu\in\mathfrak{M}^+(K) \text{ and } \textnormal{\textbf{V}}_{\alpha,q}\mu \leq 1 \textnormal{ on } \textnormal{spt}(\mu)  \right\}.
			\end{equation} 
	\end{proposition}
Note that by Fubini's theorem we have 
	\begin{equation}
		(\dM\mu)(x,t)=\int_{\mathbb{R}^{d+1}}({\Gamma}^{\alpha}\ast {\Gamma}^{\alpha})(x-z,t-\tau)d\mu(z,\tau)\nonumber.
	\end{equation} 
Let  $0<\alpha,\beta <n$ be such that  $\alpha+\beta<n$.  By identity \eqref{syb-Gamma1} one has  ${\Gamma}^{\alpha}\ast {\Gamma}^{\beta}= {\Gamma}^{\alpha+\beta}$. Therefore, $\dM\mu={\Gamma}^{2\alpha}\mu$. Similarly,  
			\begin{equation}
				(\textnormal{\textbf{V}}_{\alpha,2}\mu)(x,t)=\int_{\mathbb{R}^{d+1}}(\mathscr{G}_{\alpha}\ast \mathscr{G}_{\alpha})(x-z,t-\tau)d\mu(z,\tau)=(\mathscr{G}_{2\alpha}\mu)(x,t) .\nonumber
			\end{equation}
In particular, the  thermal capacity $\textnormal{cap}^{\mathcal{T}}$ (see e.g \cite{Lanconelli, EG}) of compact sets $K$ can be characterized by:
\begin{equation}
	\textnormal{cap}^{\mathcal{T}}(K)=\dot{C}_{1,2}(K)=\sup\{\mu(K):\mu \in \mathfrak{M}^{+}(K),\;\Gamma\ast \mu \leq 1\text{ on }\mathbb{R}^{d+1}\}.\nonumber
\end{equation}
In general, we have the following corollary. 
\begin{corollary}\label{Cap-and-Cap_T}Let $K\subset \mathbb{R}^{d+1}$ be compact. 
\begin{itemize}
	\item[(i)] Assume  $\dot{C}_{\alpha,2}(K)<\infty$, for all  $0<\alpha<n/2$. Then, 
 \begin{equation}
		\dot{C}_{\alpha,2}(K) = \sup\{\mu(K):\mu \in \mathfrak{M}^{+}(K)\mbox{ and }{\Gamma}^{2\alpha}\ast \mu \leq 1\, \text{ on }\, \text{spt}(\mu)\}.\nonumber
	\end{equation}
	\item[(ii)] Assume  $C_{\alpha,2}(K)<\infty$, for all  $\alpha>0$. Then, 
	\begin{equation}
		   C_{\alpha,2}(K) = \sup\{\mu(K):\mu \in \mathfrak{M}^{+}(K)\mbox{ and }\mathscr{G}_{2\alpha}\ast \mu \leq 1\, \text{ on }\, \text{spt}(\mu)\}.\nonumber
	   \end{equation}
\end{itemize}	
\end{corollary}

\section{Parabolic dyadic and continuous Wolff's inequality}\label{Wolff-inequality}
In this section we establish the parabolic dyadic Wolff's inequality and as a byproduct we obtain its  continuous version. 

\subsection{Dyadic Wolff's inequality} \label{dyadic-wolff}


Let $Q=Q(x,\ell) \subset \mathbb{R}^d$ be an open cube with side parallel to coordinate axes and sidelength $\ell_Q$. 
The open parabolic rectangle $R((x,t),\ell)\subset \mathbb{R}^{d+1}$ with sidelength $\ell=\ell_R=\ell_Q$ and center $(x,t)$ is defined by 
\begin{equation}
	R((x,t),\ell)=Q\times (t-\ell^2,t+\ell^2),
\end{equation}
for all $t>0$. Also, we define the time-forward  and time-backward parabolic rectangle by 
\begin{align}
	R^+((x,t),\ell,\gamma)&=Q\times (t+(1-\gamma)\ell^2,t+\ell^2)\; \text{ and }\; R^{-}((x,t),\ell,\gamma)=Q\times (t-\ell^2,t-(1-\gamma)\ell^2),\nonumber
\end{align}
where $0<\gamma\leq 1$ is called the time-lag, see Figure \ref{fig:exemplo1}.
\begin{figure}[htb]
    \centering  
\begin{tikzpicture}[scale=1.4] 
    \draw[->] (-1.5,0) -- (1.5,0) node[right]{$\mathbb{R}^{d}$};
    \draw[->] (0,0) -- (0,4) node[above]{$t$};

    \draw[dashed] (-0.5,0) -- (-0.5,4);
    \draw[dashed] (0.5,0) -- (0.5,4);

    \filldraw[fill=red!20, draw=red] (-0.5,1) rectangle (0.5,1.5) node[midway, red]{$R^{-}$};

    \filldraw[fill=blue!20, draw=blue] (-0.5,2.5) rectangle (0.5,3) node[midway, blue]{$R^{+}$};

\end{tikzpicture}
\caption{Time-forward  and time-backward parabolic rectangle in $\mathbb{R}^{d+1}_+$.}
\label{fig:exemplo1}
\end{figure}
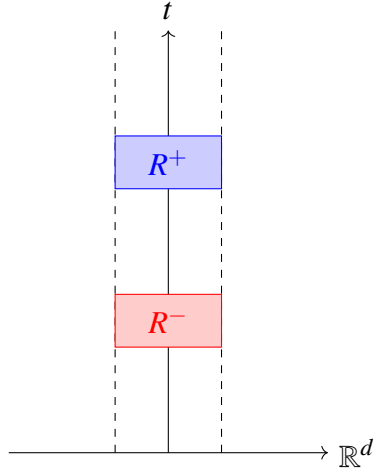
The dilation of  $R$ with respect to $\delta_{\lambda}(x,t)=(\lambda x,\lambda^2 t)$ is given by $\delta_\lambda R((x,t),\ell)=\lambda Q \times (t-\lambda^2\ell^2,\,t+\lambda^2 \ell^2)$.

Following \cite{COV} and \cite{Wolff}, the dyadic formulation of ${\mathbf{I}}_\alpha\mu$ and $\dot{\textnormal{\textbf{W}}}_{\alpha,q}\mu$ in the intrinsic parabolic setting requires a dyadic grid adapted to this geometry. To this end, we consider dyadic rectangles derived from time-backward parabolic rectangles $R^{-}=R^{-}(z,\ell, 1)$ via scaling map $(x,t) \mapsto (\lambda x, \lambda^2 t)$.  We prefer recall a construction based on the scaling $(x,t) \mapsto (\lambda x, \lambda^p t)$ for $1<p<\infty$, which naturally can lead to extensions of generalized forms of parabolic Wolff's inequalities. For this purpose, Saari \cite{Saari} constructed a dyadic grid $\mathfrak{D}(Q_0)$ with respect to $(x,t)\mapsto (\lambda x,\lambda^pt)$ for $p\in (1,\infty)$, where $Q_0\subset\mathbb{R}^{d+1}$ is a parabolic rectangle. 

\begin{lemma}[\cite{Saari}]\label{lem-Christ} Let $Q_0$ be a parabolic rectangle with sidelength $\ell(Q_0)>0$. Then, there exists boxes with properties of the dyadic tree that are almost parabolic sub-rectangles of $Q_0$ such that 
	\begin{itemize}
		\item[(i)] $\mathfrak{D}(Q_0)=\bigcup_{j=0}^{\infty} \mathfrak{D}_j(Q_0)$, where $\mathfrak{D}_j(Q_0)=\{R^{(j)}_k\}_{k}$ and $\bigcup_k R^{(j)}_k=Q_0\in\mathfrak{D}_0$.
		\item[(ii)] Given $R,R'\in \mathfrak{D}(Q_0)$ we have $R\cap R'\in \{\emptyset, R, R'\}$. For each $R\in \mathfrak{D}_j(Q_0)$, there is a unique $R'\in\mathfrak{D}_{j-1}(Q_0)$ such that $R'\supset R$ and $\vert R\vert\simeq \vert R'\vert$. 
		\item[(iii)] For every $R\in\mathfrak{D}_j(Q_0)$, there is a parabolic rectangle $R'$ obtained from $Q_0$ via dilation $(x,t)\mapsto (\lambda x,\lambda^p t)$ and translation such that $\ell(R')=2^{-j}\ell(Q_0)$, $R'\supset R$, and $\vert R\vert\simeq \vert R'\vert$, for all $1<p<\infty$. 
	\end{itemize}
\end{lemma}

Note that  $\mathfrak{D}(R^{-})$ covers only a portion of the space  $\mathbb{R}^{d+1}_+$. To construct a tiling of $\mathbb{R}^{d+1}_+$ with dyadic rectangles intrinsically linked to the parabolic geometry and time-backward structures, one may consider the dyadic lattice $\mathscr{D}$ introduced by Lerner and Nazarov \cite{Lerner}. Namely, a  parabolic dyadic lattice $\mathscr{D}$ in the sense of Lerner and Nazavov is defined as follows: 
\begin{definition}\label{par-dyad-lattice}A parabolic dyadic lattice $\mathscr{D}$ in $\mathbb{R}^{d+1}_+$ is any collection of time-backward rectangles $Q$ such that 
    \begin{enumerate}
        \item[(i)] if $Q\in\mathscr{D}$, then $\mathfrak{D}(Q)\in\mathscr{D}$, that is, each child of $Q$ obtained from Lemma \ref{lem-Christ} is in $\mathscr{D}$.  
        \item[(ii)] if $Q', Q''\in\mathscr{D}$ then there exists $Q\in\mathscr{D}$ such that $Q', Q''\in\mathfrak{D}(Q)$. 
        \item[(iii)] for every compact set $K\subset \mathbb{R}^{d+1}_+$, there exists $Q\in\mathscr{D}$ such that $K\subset Q$.
    \end{enumerate} 
\end{definition}
Given a time-backward rectangle $R^{-}\in\mathscr{D}$, by Definition \ref{par-dyad-lattice}(ii) we can split $\mathscr{D}$ into generations $\mathscr{D}_k$, where each rectangle $R$ has the sidelength given by $\ell(R)=2^{-k}\ell(R^{-})$ for all $k\in \mathbb{Z}$. Now invoking Definition \ref{par-dyad-lattice}(iii) and employing a trick with Definition \ref{par-dyad-lattice}(ii), then $\mathscr{D}_k$ tile $\mathbb{R}^{d+1}_+$.

Let $\mu \in \mathfrak{M}^+(\mathbb{R}^{d+1})$ and let $\mathbf{I}_{\alpha}\mu$ its parabolic potential defined by  
\begin{equation}
	(\mathbf{I}_{\alpha}\mu)(z)
	=\int_{0}^{\infty}\frac{\mu(Q_s(z))}{s^{n-\alpha}}\frac{ds}{s},\nonumber
\end{equation}
for $0<\alpha<n$. It is well known (see \cite[Proposition~4.8]{Nguyen}) that 
\begin{equation}
	\Vert \Gamma^\alpha \mu \Vert_{L^{q}(\mathbb{R}^{d+1})}\simeq \Vert {\mathbf{I}}_{\alpha}\mu\Vert_{L^{q}(\mathbb{R}^{d+1})}\simeq \Vert \check{\Gamma}^\alpha \mu \Vert_{L^{q}(\mathbb{R}^{d+1})}.
\end{equation}

We are now in a position to formalize a dyadic version of the previous parabolic potential. For $\mu\in\mathfrak{M}^+(\mathbb{R}^{d+1})$ the parabolic dyadic decomposition of ${\mathbf{I}}_\alpha\mu$ is defined by 
\begin{equation}
	(\textnormal{\textbf{I}}^{\mathscr{D}}_\alpha\mu)(y,s) = \sum_{\;\;\;\,R\in\mathscr{D}}\ell_{R}^{\alpha-n}\mu(R)\mathds{1}_R(y,s). 
\end{equation}
The (homogeneous) parabolic dyadic Wolff potential $\dot{\textnormal{\textbf{W}}}_{\alpha,q}^{\mathcal{D}}\mu$ is defined  by 
\begin{equation}
	\dot{\textnormal{\textbf{W}}}_{\alpha,q}^{\mathscr{D}}\mu(x,t)=\sum_{\;\;\;\,R\in\mathscr{D}}\left(\ell_R^{\alpha q-n}\mu(R)\right)^{q'-1}\mathds{1}_R(x,t),
\end{equation}
where the sum runs over all time-backward parabolic dyadic rectangles $R\in\mathscr{D}_j$ with sidelength $\ell(R)=2^{-j}\ell(R^{-})$ for  $j=0,\pm 1,\pm 2, \cdots$. For each $\delta\in (0,\infty)$,  we have known from  \cite[Proposition 4.8]{Nguyen} that 
\begin{equation}\label{equiv-riesz-bessel-Lp}
\Vert \mathscr{G}_\alpha \mu \Vert_{L^q(\mathbb{R}^{d+1})} \simeq \Vert \mathbf{I}^{\delta}_\alpha\mu\Vert_{L^q(\mathbb{R}^{d+1})}\simeq \Vert \check{\mathscr{G}}_\alpha \mu \Vert_{L^q(\mathbb{R}^{d+1})}
\end{equation}
for all $1<q<\infty$ and $0<\alpha<n$, where the truncated $\mathbf{I}^{\delta}_{\alpha}$ potential is given by 
\begin{equation}
	(\mathbf{I}^{\delta}_{\alpha}\mu)(z)
	=\int_{0}^{\delta}\frac{\mu(Q_s(z))}{s^{n-\alpha}}\frac{ds}{s}.\nonumber
\end{equation}
 This motivated to define the parabolic dyadic version of ${\mathscr{G}}_{\alpha}\mu$ as follows 
\begin{equation}
	({\mathscr{G}}_{\alpha}^{\mathscr{D}}\mu)(y,s) = \sum_{\;\;\;\,R\in\mathscr{D},\;  \ell(R)<1}\ell_{R}^{\alpha-n}\mu(R)\mathds{1}_R(y,s) \nonumber,
\end{equation}
and the (non-homogeneous) parabolic dyadic Wolff potential $\Wolff^{\mathscr{D}}\mu$ is defined by 
\begin{equation}
	(\Wolff^{\mathscr{D}}\mu)(x,t)=\sum_{\;\;\;\,R\in\mathscr{D}, \; \ell(R)<1}\left(\ell_R^{\alpha q-n}\mu(R)\right)^{q'-1}\mathds{1}_R(x,t)\nonumber,
\end{equation}
where the sum runs over  parabolic dyadic rectangles $R\in\mathscr{D}_j$ with $\ell_R<1$.  Let us define the homogeneous parabolic dyadic energy of $\mu \in \mathfrak{M}(\mathbb{R}^{d+1})$ and, respectively, its non-homogeneous variant by
$$
\dot{\textnormal{E}}^{\mathscr{D}}_{\alpha,q}\mu =\int_{\mathbb{R}^{d+1}} \vert{\mathbf{I}}^{\mathscr{D}}_\alpha\mu\vert^{q'}d\mathscr{L} \quad \text{and}\quad 	{\textnormal{E}}^{\mathscr{D}}_{\alpha,q}\mu =\int_{\mathbb{R}^{d+1}} \vert {\mathscr{G}}_{\alpha}^{\mathscr{D}}\mu\vert^{q'}d\mathscr{L}, 
$$
where $1<q<\infty$ and $0<\alpha<n$.

\begin{theorem}[Parabolic dyadic Wolff's inequality]\label{gen-Wolff-ineq}  Let $\mu \in \mathfrak{M}(\mathbb{R}^{d+1})$ and $1<q<\infty$. If $\alpha q<n$, then 	
	\begin{equation}\label{dya-version1}
		\dot{\textnormal{E}}^{\mathcal{D}}_{\alpha,q}\mu \simeq \int_{\mathbb{R}^{d+1}} (\dot{\textnormal{\textbf{W}}}_{\alpha,q}^{\mathcal{D}}\mu)d\mu.
	\end{equation}
	If $0<\alpha\leq n/q$, then 
	\begin{equation}\label{trunc-dya-version2}
		\textnormal{E}^{\mathcal{D}}_{\alpha,q}\mu \simeq \int_{\mathbb{R}^{d+1}} (\Wolff^{\mathcal{D}}\mu)d\mu.
	\end{equation}
\end{theorem}

The proof of  Theorem \ref{gen-Wolff-ineq} relies on a key parabolic dyadic version of the  \cite[Proposition 2.2]{COV}. Indeed, let $1\leq s<\infty$ and   $\Lambda = \{\lambda_{R}\}_{R\in \mathscr{D}}$ be a sequence of nonnegative numbers. Let $\sigma\in\mathfrak{M}^+(\mathbb{R}^{d+1})$ be  a positive measure in $\mathbb{R}^{d+1}$ and set 
\begin{align*}
	&A_{1}(\Lambda)=\int_{\mathbb{R}^{d+1}}\Big[\sum_{R \in \mathscr{D}}\dfrac{\lambda_{R}}{\sigma(R)}\mathds{1}_{R}(x,t)\Big]^{s}d\sigma(x,t)\\ 
	&A_{2}(\Lambda) = \sum_{R \in \mathscr{D}}\lambda_{R}\Big(\frac{1}{\sigma(R)}\sum_{R^{'}\subset R}\lambda_{R^{'}}\Big)^{s-1}\\
	&A_{3}(\Lambda)=\int_{\mathbb{R}^{d+1}}\Big[\sup_{(x,t)\in R}\Big(\frac{1}{\sigma(R)}\sum_{R^{'}\subset R}\lambda_{R^{'}}\Big)\Big]^{s}d\sigma(x,t), \quad \text{for all} \quad R\in \mathscr{D}.	
\end{align*}

\begin{lemma}[\cite{COV}]\label{olema1} If $\sigma\in\mathfrak{M}^+(\mathbb{R}^{d+1})$ is a Borel measure and $1<s<\infty$, then $A_{1}(\Lambda)$, $A_{2}(\Lambda)$ and $A_{3}(\Lambda)$ are equivalents, since 
	\begin{align}
			&A_{1}(\Lambda) \lesssim_{s}A_{2}(\Lambda)\lesssim_s\;A_{3}(\Lambda)\leq \Vert \textnormal{\textbf{M}}_{\sigma}^{\mathscr{D}}\Vert_{L^{s}(\sigma)\rightarrow L^{s}(\sigma)}\,A_{1}(\Lambda),
		\end{align}
		where $\textnormal{\textbf{M}}_{\sigma}^{\mathscr{D}}\mu$ is the dyadic Hardy-Littlewood maximal function defined by
		\begin{equation}
			(\textnormal{\textbf{M}}_{\sigma}^{\mathscr{D}}\mu)(x,t)=\sup_{R\in\mathscr{D}}\frac{\mu(R)}{\sigma(R)}
		\end{equation}
		with supremum taken over all time-lag parabolic dyadic rectangle $R$ containing $(x,t)$. 
\end{lemma}

The continuously in $L^s(\sigma)$  of the dyadic maximal function $\textnormal{\textbf{M}}_{\sigma}^{\mathscr{D}}\mu$ is fundamental in the proof Lemma \ref{olema1}. If fact, we have $\textnormal{\textbf{M}}_{\sigma}^{\mathscr{D}}\mu$ with $d\mu =\vert f\vert d\sigma $ is continuous on $L^s(\sigma)$ for $1<s\leq\infty$ (see e.g. \cite[Theorem 15.1]{Lerner}).
Let $K:\mathscr{D}\rightarrow [0, +\infty)$ be a function, we define the parabolic dyadic kernel by 
$$K_{\mathscr{D}}((x,t),(y,s))=\sum_{R\in \mathscr{D}}K(R)\mathds{1}_R(x,t)\mathds{1}_R(y,s)$$ for all $(x,t), (y,s)\in\mathbb{R}^{d+1}_+$.  Let $\mu\in\mathfrak{M}(\mathbb{R}^{d+1})$, the dyadic operator $T_{\mathscr{D}}\mu$ associated to $K_{\mathscr{D}}$ is given by 
\begin{equation}
(T_{\mathscr{D}}\mu)(x,t)=\sum_{R\in\mathscr{D}} K(R)\mu(R)\mathds{1}_R(x,t).
\end{equation}
Note that the dyadic kernel of the potential $\mathbf{I}_\alpha^\mathscr{D}\mu$ depends only on the size of $R$, namely, $K(R) =\ell_{R}^{\alpha-n}$. In general, let us define the average of the kernel $K_{\mathscr{D}}$ with respect to doubling measure $\sigma$ as follows   
\begin{equation}
	\overline{K}(R)(x,t)=\frac{1}{\sigma(R)}\sum_{R'\subset R}K(R')\sigma(R')\mathds{1}_{R'}(x,t).\nonumber
\end{equation} 
 Now from \cite[Proposition 2.4(i)]{COV} we have:  
\begin{proposition}[\cite{COV}]\label{Radial} If $d\sigma=d\mathcal{L}$ and there exists a nonincreasing function $f$ such that $K(R)=f(\ell_R)$ for each $R\in\mathscr{D}$, then   
	\begin{equation}
		\overline{K}(R) \simeq \ell_{R}^{-n}\int_{0}^{\ell_{R}}f(s)s^{n-1}ds.
	\end{equation}
\end{proposition}

\begin{proof}[Proof of Theorem \ref{gen-Wolff-ineq}]  Notice that the kernel of ${\mathbf{I}}_{\alpha}^{\mathscr{D}}\mu$ is given by $K(R) = \ell_R^{\alpha-n}$, we set
$\lambda_R := K(R)\, \mu(R)\,\sigma(R)$ and $s = q'$. Then, by Proposition \ref{Radial} we obtain 
	\begin{align}
		A_2(\{\lambda_{R}\})= \sum_{R \in \mathscr{D}} \lambda_R \left(\frac{1}{\sigma(R)}\sum_{R'\subset R}\lambda_{R'}\right)^{q'-1}&= \sum_{R \in \mathscr{D}} K(R) \mu(R)\sigma(R) \left(\int_{R}\bar{K}(R)d\mu\right)^{q'-1}\nonumber\\
		&= \int_{\mathbb{R}^{d+1}}\sum_{R \in \mathscr{D}} K(R)\sigma(R) \left(\int_{R}\bar{K}(R)d\mu\right)^{q'-1}\mathds{1}_Rd\mu \nonumber\\
		&=\int_{\mathbb{R}^{d+1}}\sum_{R \in \mathscr{D}} \ell_R^{\alpha-n}\vert R\vert\left(\int_{R}\ell_R^{\alpha-n}d\mu\right)^{q'-1}\mathds{1}_R d\mu \nonumber\\
		&=\int_{\mathbb{R}^{d+1}}(\dot{\textbf{W}}^{\mathscr{D}}_{\alpha,q}\mu)d\mu.\label{Wolff2}
	\end{align}
Note that $A_{1}(\{\lambda_R\})=\dot{\textnormal{E}}^{\mathscr{D}}_{\alpha,q}\mu$. Moreover, we have $A_3(\{\lambda_R\})= \int_{\mathbb{R}^{d+1}}(\textbf{M}_{K,\sigma}^{\mathcal{D}}\mu)^{q'}d\sigma$, where $\textbf{M}_{K,\sigma}^{\mathcal{D}}\mu$ is defined by 
\begin{equation}
	(\textbf{M}_{K,\sigma}^{\mathcal{D}}\mu)(x,t) = \sup_{(x,t)\in R} \frac{1}{\sigma(R)}\left(\sum_{R'\subset R}K(R')\sigma(R')\mu(R')\right)\nonumber.
\end{equation}
Applying Lemma \ref{olema1} immediately gives \eqref{dya-version1}. The proof of \eqref{trunc-dya-version2} can be derived similarly. 
\end{proof}

\subsection{Continuous Wolff's inequality}
In this section we establish the continuous version of the parabolic dyadic Wolff's inequality. To this end, let us define the continuous version of the previous parabolic dyadic nonlinear energy of $\mu\in\mathfrak{M}(\mathbb{R}^{d+1})$, 
\begin{equation}
	\dot{\textnormal{E}}_{\alpha,q}\mu =\int_{\mathbb{R}^{d+1}} \vert\cdGama\mu\vert ^{q'}d\mathscr{L} \quad \text{and}\quad 	{\textnormal{E}}_{\alpha,q}\mu =\int_{\mathbb{R}^{d+1}} \vert \check{\mathscr{G}}_{\alpha}\mu\vert^{q'}d\mathscr{L},
\end{equation} 
where $1<q<\infty$ and $0<\alpha<n$.  

\begin{theorem}[Continuous parabolic Wolff's inequality] \label{cont-Wolff-ineq}For  $\mu\in\mathfrak{M}^+(\mathbb{R}^{d+1})$ and $1<q<\infty$, we have:  
	\begin{itemize}
		\item[(i)] If $0<\alpha<n/q$, then  
	\begin{equation}
			\dot{\textnormal{E}}_{\alpha,q}\mu \simeq \int_{\mathbb{R}^{d+1}} (\dot{\textnormal{\textbf{W}}}_{\alpha,q}\mu)d\mu.\nonumber
		\end{equation}
		\item[(ii)] If $0<\alpha\leq n/q$, then  
	\begin{equation}
			\textnormal{E}_{\alpha,q}\mu \simeq \int_{\mathbb{R}^{d+1}} ({\textnormal{\textbf{W}}}_{\alpha,q}\mu)d\mu.\nonumber
		\end{equation}
	\end{itemize}
	
\end{theorem}
To prove Theorem \ref{cont-Wolff-ineq}, it is sufficient to establish the following inequalities:
\begin{equation}\label{ineqs-wolff-cont}
	\dot{\textnormal{E}}_{\alpha,q}\mu \lesssim \int_{\mathbb{R}^{d+1}} (\dot{\textnormal{\textbf{W}}}_{\alpha,q}\mu)d\mu \quad \text{ and } \quad \textnormal{E}_{\alpha,q}\mu \lesssim \int_{\mathbb{R}^{d+1}} ({\textnormal{\textbf{W}}}_{\alpha,q}\mu)d\mu.
\end{equation}
Indeed, we first introduce the truncated parabolic Wolff potential   
\begin{equation}\label{Wolff-trucated}
(\mathbf{W}^{\delta}_{\alpha,q}\mu)(z)
= \int_{0}^{\delta}\left(\frac{\mu(Q_{r}(z))}{r^{n-\alpha q}}\right)^{q'-1}\frac{dr}{r},
\end{equation}
for $0<\delta\leq \infty$. If we assume that $\textbf{V}_{\alpha,q}^{\delta}\mu(x,t) \gtrsim {\textnormal{\textbf{W}}}_{\alpha,q}^{\delta}\mu(x,t)$ holds for $(x,t)\in\mathbb{R}^{d+1}$ and $\delta\in (0,\infty]$ (see Lemma \ref{easy-part}), then by Fubini's theorem we have 
\begin{equation}
	\int_{\mathbb{R}^{d+1}} (\dot{\textnormal{\textbf{W}}}_{\alpha,q}\mu)d\mu \lesssim \int (\dHM\mu)d\mu = \int_{\mathbb{R}^{d+1}} \vert\cdGama\mu\vert ^{q'}d\mathcal{L}=\dot{\textnormal{E}}_{\alpha,q} \mu.\nonumber
\end{equation}
Now by equivalence $\Vert {\mathbf{I}}^{\delta}_{\alpha}\mu\Vert_{L^{q'}(\mathbb{R}^{d+1})}\simeq \Vert \check{\mathscr{G}}_\alpha \mu \Vert_{L^{q'}(\mathbb{R}^{d+1})}$ for each $0<\delta<\infty$, we have
\begin{equation}
		\int (\textnormal{\textbf{W}}^{\delta}_{\alpha,q}\mu)d\mu \lesssim \int (\textbf{V}_{\alpha,q}^{\delta}\mu)d\mu =\int \vert {\mathbf{I}}^{\delta}_{\alpha}\mu\vert^{q'} d\mu \simeq \textnormal{E}_{\alpha,q} \mu.\nonumber
\end{equation}

\begin{proof}[Proof of Theorem \ref{cont-Wolff-ineq}] 
 
	Splitting $\mathscr{D}$ into $j-$generations $\mathscr{D}_j$ we can estimate 
	\begin{align}
        \mathbf{I}_{\alpha}\mu(y,s) =\int_0^{\infty}\frac{\mu(B_{\rho}(y)\times [s-\rho^2,s))}{\rho^{n-\alpha}}\frac{d\rho}{\rho}
	&\simeq \sum_{j\in\mathbb{Z}}(2^{-j}\ell_0)^{\alpha-n}\mu(Q_{2^{-j}\ell_0}(y,s))\nonumber\\
	& 
	\lesssim \sum_{\;\;R\in\mathscr{D}}\ell_R^{\alpha-n}\mu(\tilde{R})\mathds{1}_{R}(y,s)\nonumber\\
	&\lesssim\sum_{\;\;\tilde{R}\in\widetilde{\mathscr{D}}}\ell_{\tilde{R}}^{\alpha-n}\mu(\tilde{R})\mathds{1}_{R}(y,s)=\textnormal{\textbf{I}}_{\alpha}^{\mathcal{D}}\mu(y,s),\nonumber
    \end{align} 
	since by Three Lattice Theorem \cite[Theorem 3.1]{Lerner} one has $\widetilde{\mathscr{D}}=\bigcup_{j=1}^{3^{n}}\mathscr{D}^{(j)}$, where each $\mathscr{D}^{(j)}$ is a dyadic lattice in the sense of  Definition \ref{par-dyad-lattice}. Hence, the rectangles $\tilde{R}=\delta_3R\in\widetilde{\mathscr{D}}$ can be seen as parabolic dyadic rectangles.
Note that,  
	\begin{align}
		\dot{\textnormal{\textbf{W}}}_{\alpha,q}\mu(x,t)&\simeq \sum_{j\in\mathbb{Z}}\left((2^{-j}\ell_0)^{\alpha q-d-2}\mu(B_{2^{-j}\ell_0}(x)\times [t-4^{-j}\ell^2_0,\,t))\right)^{{q'}-1}\nonumber.
	\end{align}
Then, we have 
\begin{align}
	\dot{\textnormal{\textbf{W}}}_{\alpha,q}\mu(x,t)
	&\gtrsim  \sum_{\;\;R\in\mathscr{D}}\left(\ell_R^{\alpha q-n}\mu(R)\right)^{{q'}-1}\mathds{1}_{R}(x,t)
	= \dot{\textnormal{\textbf{W}}}_{\alpha,q}^{\mathcal{D}}\mu (x,t) .\label{Riesz-disc3}
\end{align}
Therefore, by  previous inequality and Theorem \ref{gen-Wolff-ineq} we obtain  
	\begin{align*}
		\int_{\mathbb{R}^{d+1}}(\dot{\textnormal{\textbf{W}}}_{\alpha,q}\mu)d\mu \gtrsim  \int_{\mathbb{R}^{d+1}}(\dot{\textnormal{\textbf{W}}}_{\alpha,q}^{\mathscr{D}}\mu)d\mu\simeq\dot{\text{E}}^{\mathscr{D}}_{\alpha,q}\mu&=\int_{\mathbb{R}^{d+1}}\vert {\mathbf{I}}_{\alpha}^{\mathscr{D}}\mu\vert^{q'}d\mathscr{L}\\
        &\gtrsim  \int_{\mathbb{R}^{d+1}}\vert {\mathbf{I}}_{\alpha}\mu\vert^{q'}d\mathscr{L}\simeq \dot{\textnormal{E}}_{\alpha,q}\mu,
	\end{align*}
for all $1<q<n/\alpha$, since $\Vert {\mathbf{I}}_{\alpha}\mu\Vert_{L^{q'}(\mathbb{R}^{d+1})}\simeq \Vert \check{\Gamma}^\alpha \mu \Vert_{L^{q'}(\mathbb{R}^{d+1})}$ (see \cite[Proposition 4.8]{Nguyen}).  Theorem \ref{cont-Wolff-ineq}(ii) can be established proceeding along similar lines. Indeed,  
\begin{align*}
	\int_{\mathbb{R}^{d+1}}({\textnormal{\textbf{W}}}_{\alpha,q}\mu)d\mu \gtrsim  \int_{\mathbb{R}^{d+1}}({\textnormal{\textbf{W}}}_{\alpha,q}^{\mathcal{D}}\mu)d\mu\simeq\int_{\mathbb{R}^{d+1}}\vert {\mathscr{G}}_{\alpha}^{\mathscr{D}}\mu\vert^{q'}d\mathscr{L}\gtrsim \int_{\mathbb{R}^{d+1}}\vert {\mathbf{I}}_{\alpha}^{\delta}\mu\vert ^{q'}d\mathscr{L}\simeq \textnormal{E}_{\alpha,q}\mu ,
\end{align*}
as required. 
\end{proof}	
It remains to verify the inequality $\textnormal{\textbf{V}}_{\alpha,q}^{\delta}\mu(x,t) \gtrsim {\textnormal{\textbf{W}}}_{\alpha,q}^{\delta}\mu(x,t)$ for $(x,t)\in\mathbb{R}^{d+1}$ and $\delta\in (0,\infty]$. The issue of the following lemma.
\begin{lemma} \label{easy-part}Let $\mu\in\mathfrak{M}^+(\mathbb{R}^{d+1})$ and $0<\alpha<n$, then $\textnormal{\textbf{V}}_{\alpha,q}^{\delta}\mu(x,t) \gtrsim {\textnormal{\textbf{W}}}_{\alpha,q}^{\delta}\mu(x,t)$ for all $(x,t)\in\mathbb{R}^{d+1}$.
\end{lemma}
\begin{proof} For each $(y,s)\in Q_{2^{-i-1}\delta}(x,t)$ we have $Q_{2^{-i}\delta}(y,s)\supseteq Q_{2^{-i-1}\delta}(x,t)$.  Indeed, let $(z,\tau)\in Q_{2^{-i-1}\delta}(x,t)$ then
\begin{align}
	\vert z-y\vert +\vert \tau-s\vert^{\frac{1}{2}}\leq (\vert z-x\vert +\vert \tau-t\vert^{\frac{1}{2}})+(\vert x-y\vert +\vert t-s\vert^{\frac{1}{2}})< 2^{-i}\delta.\nonumber
\end{align}
Therefore,  defining the measure $d\nu(y,s) =({\mathbf{I}}_{\alpha}^{\delta}\mu)^{{q'}-1}dyds$ for $0<\delta<\infty$, we can estimate 
\begin{align}
		\mathbf{V}_{\alpha,q}^{\delta}\mu(x,t)&:=\int_0^{\delta}\frac{\nu (B_{\rho}(x)\times [t-\rho^2,t))}{\rho^{n-\alpha}}\frac{d\rho}{\rho}\nonumber\\
		&\simeq \sum_{i=0}^{\infty}(2^{-i}\delta)^{\alpha-n}\nu(Q_{2^{-i}\delta}(x,t))\nonumber\\
		&\geq \sum_{i}(2^{-i}\delta)^{\alpha-n}\int_{Q_{2^{-i-1}\delta}(x,t)} [{\mathbf{I}}_{\alpha}^{\delta}\mu(y,s)]^{{q'}-1}dyds\nonumber\\
		&\geq \sum_{i} (2^{-i}\delta)^{\alpha-n}(2^{-i}\delta)^{(\alpha-n)({q'}-1)} \int_{Q_{2^{-i-1}\delta}(x,t)} \mu(Q_{2^{-i}\delta}(y,s))^{{q'}-1}dyds\nonumber\\
		&\geq \sum_{i} (2^{-i}\delta)^{(\alpha-n){q'}} \mu(Q_{2^{-i-1}\delta}(x,t))^{{q'}-1}m(Q_{2^{-i-1}\delta}(x,t))\nonumber\\
		&=c \sum_{i=0}^{\infty} \left((2^{-i}\delta)^{\alpha q-n} \mu(B_{2^{-i}\delta}(x)\times [s-4^{-i}\delta^2,\,s))\right)^{{q'}-1}\nonumber\\
		&\simeq \textbf{W}_{\alpha,q}^{\delta}\mu(x,t).\nonumber
\end{align}
The case $\delta = \infty$ follows by an analogous argument. This completes the proof. 
\end{proof}

\subsection{Regularized dyadic Wolff's potential} 
The dyadic nonlinear potential  ${\mathbf{W}}^{\mathcal{D}}_{\alpha,q}\mu$ is not lower semicontinuous on $\mathbb{R}^{d+1}$. It can be regularized in the following way 
  \begin{equation}\label{reg-dyad-wolff}
 	({\mathcal{W}}^{\mathcal{D}}_{\alpha,q}\mu)(x,t)=\sum_{\substack{R\in\mathscr{D},\;\ell_R<1}}\left(\ell_R^{\alpha q-n}\mu(\eta_R) \right)^{{q'}-1}\,\eta_R(x,t) 
 \end{equation}
for all $\mu\in \mathfrak{M}^+(\mathbb{R}^{d+1})$, where $\eta_R\in C_c^{\infty}(\mathbb{R}^{d+1})$ such that   $\mathds{1}_R\leq \eta_R\leq \mathds{1}_{\delta_3R}$ and  $\mu(\eta_R):=\int_{\mathbb{R}^{d+1}} \eta_R d\mu$. Then, in view of Theorem \ref{gen-Wolff-ineq}, the regularized dyadic nonlinear energy can be defined by 
\begin{equation}\label{d-reg-energy}
	{\mathcal{E}}\mu=\int_{\mathbb{R}^{d+1}} \big({\mathcal{W}}^{\mathcal{D}}_{\alpha,q}\mu\big)d\mu
	=\sum_{\substack{R\in\mathscr{D},\;\ell_R<1}}b(R)[\mu(\eta_R)]^{q'},
\end{equation}
where $b(R):=\ell_R^{(\alpha q-n)({q'}-1)}$, for all $0<\alpha<n$ and $1<q\leq n/\alpha$. Using Theorem \ref{gen-Wolff-ineq} and Theorem \ref{cont-Wolff-ineq}(ii), we have the following equivalences.

\begin{lemma} \label{key-equi-energias}Let $0<\alpha<n$ and  $1<q\leq n/\alpha$.  For $\mu\in\mathfrak{M}^+$  we have 
		$${\textnormal{E}}_{\alpha,q}^{\mathcal{D}}\mu \simeq {\mathcal{E}}\mu\simeq  {\textnormal{E}}_{\alpha,q}\mu.
		$$
\end{lemma}
\begin{proof} For each compact set $K$, there exists a parabolic dyadic rectangle $R\in\mathscr{D}$ containing $K$ (see Definition \ref{par-dyad-lattice}). 
Therefore,  
	\begin{align}
		({\textbf{W}}^{\mathcal{D}}_{\alpha,q}\mu)(x,t) \lesssim  ({\textbf{W}}_{\alpha,q}\mu)(x,t) &\simeq \sum_{i=0}^{\infty}\left((2^{-i}\ell_0)^{\alpha q-n}\mu(B_{2^{-i}\ell_0}(x)\times [t-(2^{-i}\ell_0)^2,\,t))\right)^{{q'}-1}\nonumber\\
		&\lesssim \sum_{\ell_R<1}\Big(\dfrac{\mu(R)}{\ell_{R}^{n-\alpha q}}\Big)^{{q'}-1}\mathds{1}_{R}(x,t)\label{eq-Lemma3.8}.
		\end{align} 
Hence, we get 
\begin{equation}
\int_{\mathbb{R}^{d+1}}({\textbf{W}}^{\mathcal{D}}_{\alpha,q}\mu)d\mu \lesssim \int_{\mathbb{R}^{d+1}}({\textbf{W}}_{\alpha,q}\mu) d\mu \lesssim \sum_{\ell_R<1}b(R)\mu(R)^{{q'}}.\label{energia-equiv-key1}
\end{equation}  
Now, by previous inequality, Theorems \ref{gen-Wolff-ineq} and \ref{cont-Wolff-ineq} we obtain 
\begin{equation}
   {\textnormal{E}}^{\mathcal{D}}_{\alpha,q}\mu\simeq {\textnormal{E}}_{\alpha,q}\mu.\nonumber
\end{equation}
It remains to show that ${\mathcal{E}}\mu\simeq  {\textnormal{E}}^{\mathcal{D}}_{\alpha,q}\mu$. To this end, let us recall that $\mathds{1}_R\leq \eta_R\leq \mathds{1}_{\tilde{R}}$ and ${\mathcal{E}}\mu=\displaystyle\int ({\mathcal{W}}^{\mathcal{D}}_{\alpha,q}\mu)d\mu$, where $\tilde{R}=\delta_3R$. Then, 
\begin{align*}
 \sum_{\substack{R\in \mathscr{D}\\ \ell_R<1}}b(R)\mu(R)^{q'}\leq {\mathcal{E}}\mu  
 \leq \sum_{\substack{\tilde{R}\in {\mathscr{D}}\\ \ell_R<1}} b(\tilde{R})\mu(\tilde{R})^{{q'}} 
   \end{align*}
   since by Three Lattice Theorem \cite[Theorem 3.1]{Lerner}, the rectangles $\tilde{R}$ can be seen as parabolic dyadic rectangles in the sense of Lerner-Nazarov (see Definition \ref{par-dyad-lattice}). Hence, by Theorem \ref{gen-Wolff-ineq}, we obtain ${\mathcal{E}}\mu \simeq {\textnormal{E}}^{\mathcal{D}}_{\alpha,q}\mu$ as required.
	\end{proof} 

\section{Parabolic dyadic capacity}\label{parab-dyadic-capacity}

Inspired by the dual characterization of the Bessel capacity (see Proposition \ref{duality}), the parabolic dyadic Bessel capacity of a compact set $K \subset \mathbb{R}^{d+1}$ is defined by
\begin{equation}
	\mathscr{C}_{\alpha,q}(K)^{1/q}:=\sup\left\{ \Vert\mu\Vert\,:\, \mu\in \mathfrak{M}^{+}(K) \text{ and } {\mathcal{E}}\mu\leq 1\right\}.\nonumber
\end{equation}
Throughout Chapters \ref{parab-dyadic-capacity} and \ref{chapter-five}, we assume $\alpha\in (0,n)$ and $1<q\leq n/\alpha$. By Lemma \ref{key-equi-energias} and Proposition \ref{duality}, this definition is equivalent to the parabolic Bessel capacity: $\mathscr{C}_{\alpha,q}(K)\simeq {C}_{\alpha,q}(K)$.  Moreover, the capacity $\mathscr{C}_{\alpha,q}$ can be extended to arbitrary sets in the usual way. 

\begin{lemma}\label{equiv-energy}Let  Let $K\subset \mathbb{R}^{d+1}$ be compact set and $\mathfrak{M}^{+}_1(K):=\{\mu\in\mathfrak{M}^+(K)\,:\, \Vert \mu\Vert =1\}$. Then,  
	\begin{equation}
		[\mathscr{C}_{\alpha,q}(K)]^{-1}=\inf\left\{( {\mathcal{E}}\nu)^{q-1}\,:\,\nu\in \mathfrak{M}_1^+(K)\right\},\nonumber
	\end{equation}
	
\end{lemma}
\begin{proof}Let $\mu\in \mathfrak{M}^+(K)$ be such that ${\mathcal{E}}\mu\leq 1$. Then $\nu:=\frac{\mu}{\Vert \mu\Vert}\in \mathfrak{M}^{+}_1(K)$ and we have 
	$$
	{\mathcal{E}}\nu=\mu(K)^{-q'} {\mathcal{E}}\mu\leq \mu(K)^{-q'}$$ 
	which yields $ ({\mathcal{E}}\nu)^{q/q'}\leq [\mathscr{C}_{\alpha,q}(K)]^{-1}$. Consequently, 
	\begin{equation}
		\inf_{\nu\in \mathfrak{M}_1^{+}(K)} ({\mathcal{E}}\nu)^{q-1}=\inf_{\nu\in \mathfrak{M}_1^{+}(K)} ({\mathcal{E}}\mu)^{q/q'}\leq [\mathscr{C}_{\alpha,q}(K)]^{-1}\nonumber.
	\end{equation}
	For the converse inequality, let $\nu\in \mathfrak{M}^+(K)$ be such that $\nu(K)=1$. Set $\mu:=({\mathcal{E}}\nu)^{-{1}/{q'}} {\nu}$ and note that ${\mathcal{E}}\mu=1$ and $\mu(K)= ({\mathcal{E}}\nu)^{-{1}/{q'}}$. Hence, we get   
	\begin{equation}
		[\mathscr{C}_{\alpha,q}(K)]^{1/q}\geq \Vert \mu\Vert =({\mathcal{E}}\nu)^{-{1}/{q'}}\nonumber,
	\end{equation}
	which leads to the required upper bound. 
\end{proof}

\subsection{Extremal measures}\label{sec-extremal}

\begin{definition}\label{extremal-meas}Let $A\subset \mathbb{R}^{d+1}$ be an arbitrary set. A measure  $\gamma_A\in \mathfrak{M}^+_1$ is called $\mathscr{C}_{\alpha,q}-$extremal measure for $A$, if $\operatorname{spt}(\gamma_A)\subset \overline{A}$ and $\gamma_A$ minimizes energy ${\mathcal{E}}\mu$ over $\mathfrak{M}^{+}_1(A)$ such that 
	\begin{equation}\label{Energy=Cap}
		{\mathcal{E}}\gamma_A=[\mathscr{C}_{\alpha,q}(A)]^{1-q'}.  
	\end{equation}
\end{definition}
\begin{lemma} \label{Cap-extremal-meas} Every compact set $K$ admits a $\mathscr{C}_{\alpha,q}-$extremal measure $\gamma_K$.
\end{lemma}

\begin{proof} Let $\{\mu_n\}\in \mathfrak{M}_1^+(K)$ be a sequence such that $\lim_{n\rightarrow\infty}{\mathcal{E}}\mu_n= \inf_{\mu\in \mathfrak{M}_1^+(K)} {\mathcal{E}}\mu$. But   $\mathfrak{M}^{+}_1(K)$ is compact in the weak$^{\star}-$topology of $\mathfrak{M}^+$, then we can find a minimizing measure $\gamma_K\in\mathfrak{M}_1^+(K)$ such that $\mu_{n_j}\rightharpoonup \gamma_K$ and  ${\mathcal{E}}\gamma_K\leq {\mathcal{E}}\mu$ for all $\mu\in \mathfrak{M}_1^+(K)$. Indeed, since 
$$\mu\mapsto {\mathcal{E}}\mu =\int ({\mathcal{W}}^{\mathcal{D}}_{\alpha,q}\mu)d\mu$$ is lower semicontinuous in the weak$^{\star}-$topology of $\mathfrak{M}^{+}(\mathbb{R}^{d+1})$, then 
\begin{equation}
{\mathcal{E}}\gamma_K\leq \liminf_{j\rightarrow \infty}{\mathcal{E}}\mu_{n_j}=\inf_{\mu\in \mathfrak{M}_1^+(K)} {\mathcal{E}}\mu
\end{equation}
which implies that ${\mathcal{E}}\gamma_K\leq {\mathcal{E}}\mu$. Now by Lemma \ref{equiv-energy} we get 
$$[\mathscr{C}_{\alpha,q}(K)]^{\frac{1}{q-1}} = \left(\inf_{\mu\in\mathfrak{M}^{+}_1(K)}{\mathcal{E}}\mu\right)^{-1} = [{\mathcal{E}}\gamma_{K}]^{-1}$$
which leads to \eqref{Energy=Cap} as desired.
\end{proof}
In Proposition \ref{Prop:extremal-aberto}, we extend this lemma to arbitrary sets.

\subsection{\texorpdfstring{$\mathscr{C}-$capacitary measure associated to Wolff's potential}{C-capacitary measure associated to Wolff's potential}}

In this section, we establish the existence of capacitary measures for compact sets and present several characterizations of the capacity $\mathscr{C}_{\alpha,q}$. 
\begin{definition}[Capacitary measure associated to Wolff's potential] \label{capacitary-meas-Wolff}
Let $E\subset \mathbb{R}^{d+1}$ be such that $0<\mathscr{C}_{\alpha,q}(E)<\infty$.  We say that $\mu^E\in\mathfrak{M}^+$ is a $\mathscr{C}-$capacitary measure for $E$, if its $\operatorname{spt}(\mu^E)\subset\overline{E}$ and 
	\begin{itemize}
		\item[(i)] $\mu^E(\overline{E})=\mathscr{C}_{\alpha,q}(E)$
		\item[(ii)]$({\mathcal{W}}^{\mathcal{D}}_{\alpha,q}\mu^E)(x,t)\geq 1 $   $\;\mathscr{C}_{\alpha,q}-$a.e  $\;(x,t)\in E$
		\item[(iii)]$({\mathcal{W}}^{\mathcal{D}}_{\alpha,q}\mu^E)(x,t)\leq 1 $   $\;\forall \;(x,t)\in \textnormal{spt}(\mu^E)$.
	\end{itemize} 
\end{definition}

\begin{proposition}\label{prop-equilibrium}Let  $\gamma$ be an $\mathscr{C}-$extremal measure for compact set $K$, then 
	\begin{itemize}
	\item [\textnormal{\textbf{(A)}}] $({\mathcal{W}}_{\alpha,q}^{\mathcal{D}}\gamma)(x,t)\geq {\mathcal{E}}\gamma$ \; \text{ for }\;$\quad\mathscr{C}_{\alpha,q}-\text{a.e}\quad  (x,t)\in K$,
		\item [\textnormal{\textbf{(B)}}] $({\mathcal{W}}_{\alpha,q}^{\mathcal{D}}\gamma)(x,t)\leq  {\mathcal{E}}\gamma\;$  \text{ for all }  $(x,t)\in \textnormal{spt}(\gamma)$.
	\end{itemize}
In particular,  we have 
$({\mathcal{W}}_{\alpha,q}^{\mathcal{D}}\gamma)(x,t)={\mathcal{E}}\gamma\;\;$ for $\;\;\mathscr{C}_{\alpha,q}-\text{a.e}  \;\; (x,t)\in \textnormal{spt}(\gamma)$.
\end{proposition}
\begin{proof} Let $K_0\subset K$ be a compact set such that $\mathscr{C}_{\alpha,q}(K_0)>0$ and $f(x,t):={\mathcal{E}}\gamma -({\mathcal{W}}_{\alpha,q}^{\mathcal{D}}\gamma)(x,t)>0$ for all    $(x,t)\in K_0$. Since $({\mathcal{W}}_{\alpha,q}^{\mathcal{D}}\gamma)(\cdot)$  is lower semicontinuous on $\mathbb{R}^{d+1}$, then for any $\epsilon>0$ there is $\delta>0$ such that the set $F:=\overline{B((x,t),\delta/2)}\cap K_0$ is compact with $\mathscr{C}_{\alpha,q}(F)>0$, and we have  
	\begin{equation}\label{contradiction}
		F \,\subset \;K\cap \left\{(x,t)\,:\, ({\mathcal{W}}_{\alpha,q}^{\mathcal{D}}\gamma)(x,t) <{\mathcal{E}}\gamma-\epsilon\right\}.
	\end{equation}
Let $\mu\in\mathfrak{M}^{+}_1(K)$ be such that  ${\mathcal{E}}\mu<\infty$, and set $ \mu_{\delta}:=
\delta \mu +(1-\delta)\gamma$ for all $\delta\in [0,1]$. Since $\mathfrak{M}^{+}_1(K)$ is convex  and $\gamma$ minimizes energy, then $\mu_{\delta}\in \mathfrak{M}^{+}_1(K)$ and we get ${\mathcal{E}}\gamma\leq {\mathcal{E}}\mu_\delta$ for all $\delta \in (0,1)$. We will show that this leads to a contradiction. 
Indeed, applying the mean value theorem to the function $g_R(\delta)=[\gamma(\eta_R)+\delta(\mu(\eta_R)-\gamma(\eta_R))]^{q'}$, there is $0<s<\delta$ such that  
\begin{equation}
    g_R(\delta)=g_R(0)+\delta g'_R(s)=[\gamma(\eta_R)]^{q'}+\delta q' [\mu_s(\eta_R)]^{q'-1}(\mu(\eta_R)-\gamma(\eta_R)).  \nonumber
\end{equation} 
Therefore, 
\begin{align}
{\mathcal{E}}\mu_\delta
	=\sum_{\ell_R<1} b(R)g_R(\delta)&= \sum_{\ell_R<1} b(R)\left\{[\gamma(\eta_R)]^{q'}+\delta\,{q'}[\mu(\eta_R) -\gamma(\eta_R)][\mu_{s}(\eta_R)]^{{q'}-1}\right\}\nonumber\\
	&={\mathcal{E}}\gamma+\delta p \sum_{\ell_R<1}b(R)[\mu(\eta_R) -\gamma(\eta_R)]\mu_s(\eta_R)^{{q'}-1}\label{sum1}.
\end{align}
The sum in \eqref{sum1} converges uniformly on $(0,\delta)$. In fact, its absolute value is bounded by 
\begin{align}
  \sum_{\ell_R<1}b(R)[\gamma(\eta_R)+\mu(\eta_R)]^{q'} \lesssim {\mathcal{E}}\gamma + {\mathcal{E}}\mu.\nonumber
\end{align}
Therefore, letting $\delta\rightarrow 0$ in the right-hand side of \eqref{sum1} and applying \eqref{contradiction}, we obtain:
\begin{align*}
    {\mathcal{E}}\mu_\delta &=  {\mathcal{E}}\gamma+\delta {q'} \sum_{\ell_R<1} \left(\ell_R^{\alpha q-n}\gamma(\eta_R)\right)^{{q'}-1}\mu(\eta_R)\,-\,\delta {q'} \sum_{\ell_R<1} b(R)\gamma(\eta_R)^{q'}+o(\delta)\nonumber\\
&={\mathcal{E}}\gamma+\delta {q'}\int_{\mathbb{R}^{d+1}} ({\mathcal{W}}_{\alpha,q}^{\mathcal{D}}\gamma)d\mu-\delta {q'} {\mathcal{E}}\gamma+o(\delta)\\
&={\mathcal{E}}\gamma+\delta {q'}\left\{\int_{\mathbb{R}^{d+1}} \left(({\mathcal{W}}_{\alpha,q}^{\mathcal{D}}\gamma)(x,t) -{\mathcal{E}}\gamma\right)d\mu \right\}+o(\delta)\nonumber\\
&<{\mathcal{E}}\gamma-\delta {q'}\epsilon +o(\delta)<{\mathcal{E}}\gamma,
\end{align*} 
for $\delta>0$ sufficiently small, which is a contradiction. 

Let us proof $\textbf{(B)}$.  Assume ${\mathcal{W}}_{\alpha,q}^{\mathcal{D}}\gamma (z_0)>{\mathcal{E}}\gamma$ for some  $ z_0\in \text{spt}(\gamma)$. Since $({\mathcal{W}}_{\alpha,q}^{\mathcal{D}}\gamma)(\cdot)$ is (l.s.c) in $\mathbb{R}^{d+1}$, there exists an open set $V$ containing  $ z_0 \in \text{spt}(\gamma)$ such that  
\begin{equation}\label{contrad}
    ({\mathcal{W}}_{\alpha,q}^{\mathcal{D}}\gamma)(x,t)>{\mathcal{E}}\gamma \quad \text{for all }\quad  (x,t)\in V.
\end{equation} 
Now, using the previous result $\textbf{(A)}$ and noting  that $\gamma \ll  \mathscr{C}_{\alpha,q}$ ($\mathscr{C}-$absolutely continuous), we have  $({\mathcal{W}}_{\alpha,q}^{\mathcal{D}}\gamma)(x,t)\geq {\mathcal{E}}\gamma\;$ holds $\gamma-\text{a.e} \;\text{ on }\; K$. This leads to a contradiction. Indeed, 
\begin{align*}
    {\mathcal{E}}\gamma=\int ({\mathcal{W}}_{\alpha,q}^{\mathcal{D}}\gamma)d\gamma&=\int_{V} ({\mathcal{W}}_{\alpha,q}^{\mathcal{D}}\gamma)d\gamma+\int_{\text{spt}(\gamma)\backslash V} ({\mathcal{W}}_{\alpha,q}^{\mathcal{D}}\gamma)d\gamma>{\mathcal{E}}\gamma.
\end{align*}
This concludes the proof as needed.
\end{proof} 

We are now ready to show the existence of $\mathscr{C}_{\alpha,q}-$capacitary measure for compact sets. 
\begin{theorem}\label{Thm-capacitary-meas}
    For any compact set $K\subset\mathbb{R}^{d+1}$ with $0<\mathscr{C}_{\alpha,q}(K)<\infty$, there exists a  $\mathscr{C}_{\alpha,q}-$capacitary measure $\gamma^K$. 
    \end{theorem}
\begin{proof} Let $\gamma \in \mathfrak{M}^{+}(K)$ be a $\mathscr{C}-$extremal measure for $K$ (see Lemma \ref{Cap-extremal-meas}) and consider the equilibrium measure $\gamma^{K} =\mathscr{C}_{\alpha,q}(K)\gamma$. Clearly, we have $\text{spt}(\gamma^{K}) \subset K$ and  $\gamma^{K}(K) = \mathscr{C}_{\alpha,q}(K)$. Now, by Definition \ref{extremal-meas} we have 
	\begin{align*}
		({\mathcal{W}}_{\alpha,q}^{\mathcal{D}}\gamma^{K})(z) = \sum_{\ell_R<1}\left(\ell_{R}^{\alpha q-n}\gamma^{K}(\eta_{R})\right)^{{q'}-1}\eta_{R}(z)&= [\mathscr{C}_{\alpha,q}(K)]^{{q'}-1}\sum_{\ell_R<1}\left(\ell_{R}^{\alpha q-n}\gamma(\eta_{R})\right)^{{q'}-1}\eta_{R}(z)\\		
        &=[{\mathcal{E}}(\gamma)]^{-1}({\mathcal{W}}_{\alpha,q}^{\mathcal{D}}\gamma)(z).
	\end{align*}
Therefore, by Proposition \ref{prop-equilibrium} and  Definition \ref{capacitary-meas-Wolff} we obtain the result as required.   
\end{proof}

\subsubsection{Characterizations via Wolff's potential}
In this section we derive characterizations of the parabolic dyadic capacity for compact sets, formulated in terms of the associated Wolff potential. 
\begin{theorem}\label{Thm-carac-capacity}\
	\begin{itemize}
	\item [\textnormal{\textbf{(A)}}]  If $\mu\in\mathfrak{M}^+(\mathbb{R}^{d+1})$  has finite energy  ${\mathcal{E}}\mu<\infty$ and $K\subset E_{\vartheta}=\{ (x,t)\,:\, ({\mathcal{W}}_{\alpha,q}^{\mathcal{D}}\mu)(x,t)>\vartheta\}$ is a compact set, then 
		\begin{equation}
			\mathscr{C}_{\alpha,q}(K)\leq \vartheta^{-q}{\mathcal{E}}\mu,\nonumber
		\end{equation}
for all  $\vartheta>0$. In particular, we have 
    \begin{equation}
        \mathscr{C}_{\alpha,q}(E_{\vartheta})\leq \vartheta^{-q}{\mathcal{E}}\mu.\nonumber
    \end{equation}

	\item [\textnormal{\textbf{(B)}}] Let $K\subset \mathbb{R}^{d+1}$ be a compact set, then 
\begin{equation}
			\mathscr{C}_{\alpha,q}(K)=\inf\left\{ {\mathcal{E}}\mu\,:\; \mu\in\mathfrak{M}^+ \text{ and }  {\mathcal{W}}^{\mathcal{D}}\mu\geq 1\;\;\;  \mathscr{C}_{\alpha,q}-\text{a.e on }K\right\}\nonumber.
	\end{equation}
	\item [\textnormal{\textbf{(C)}}] Let $K\subset \mathbb{R}^{d+1}$ be a compact set, then 
\begin{equation}
	\mathscr{C}_{\alpha,q}(K)= \sup \left\{ \mu(K) \,:\; \mu\in\mathfrak{M}^+(K) \text{ and } {\mathcal{W}}_{\alpha,q}^{\mathcal{D}}\mu\leq 1\;\text{ on }\;\textnormal{spt}(\mu)\right\}\nonumber.
\end{equation}
	\end{itemize}
\end{theorem}
\begin{proof} Let  $\gamma\in \mathfrak{M}^+_1(K)$ be such that  ${\mathcal{E}}\gamma<\infty$ on $\text{spt}(\gamma)\subset E_{\vartheta}$. Then, by  Holder's inequality we obtain  
	\begin{align}
		\vartheta\leq \int_{\mathbb{R}^{d+1}}{\mathcal{W}}_{\alpha,q}^{\mathcal{D}}\mu\,d\gamma& =\sum_{\ell_R<1} \left(\ell_R^{\alpha q-n}\mu(\eta_R)\right)^{q'-1}\gamma(\eta_R)\nonumber\\
		&=\sum_{\ell_R<1} \left(\ell_R^{(\alpha q-n)\frac{q'-1}{q}}[\mu(\eta_R)]^{q'-1}\right)\left(\ell_R^{(\alpha q-n)\frac{q'-1}{q'}} \gamma(\eta_R)\right)\nonumber\\
		&\leq \left( \sum_{\ell_R<1} \ell_R^{(\alpha q-n)(q'-1)}[\mu(\eta_R)]^{q'}\right)^{1/q}\left( \sum_{\ell_R<1} \ell_R^{(\alpha q-n)(q'-1)}[\gamma(\eta_R)]^{q'}\right)^{1/q'}\nonumber\\
		&\leq \left({\mathcal{E}}\mu\right)^{1/q}\left({\mathcal{E}}\gamma\right)^{1/q'}.\label{Energy-Energy}
	\end{align}
Now, from Lemma \ref{Cap-extremal-meas}, we can assume that $\gamma=\gamma_K$ is  $\mathscr{C}-$extremal measure for compact set  $K\subset E_{\vartheta}$. Then using Definition \ref{extremal-meas} and inequality \eqref{Energy-Energy} we have 
\begin{equation}
	\mathscr{C}_{\alpha,q}(K) =({\mathcal{E}}\gamma_K)^{-q/q'}\leq  \vartheta^{-q}{\mathcal{E}}\mu.\nonumber
\end{equation}
Let us proof \textbf{(B)}.  For every compact set $K\subset E_1$, the item \textbf{(A)} leads to $\mathscr{C}_{\alpha,q}(K) \leq {\mathcal{E}}\mu$. In particular, 
\begin{equation}
	\mathscr{C}_{\alpha,q}(K) \leq \inf\left\{ {\mathcal{E}}\mu\,:\; \mu\in\mathfrak{M}^+ \text{ and }  {\mathcal{W}}_{\alpha,q}^{\mathcal{D}}\mu\geq 1\;\;\;  \mathscr{C}-\text{a.e on }K\right\}\nonumber.
\end{equation}
The infimum mentioned above is precisely reached at $\mu^K$, the $\mathscr{C}-$capacitary measure for $K$. In fact, it can be demonstrated that
\begin{equation}
	{\mathcal{E}}\mu^K = \mu^K(K)=\mathscr{C}_{\alpha,q}(K). \label{key-cap-dya-ener}
\end{equation} 
Let us proof \textbf{(C)}. It suffices to show that $\mu(K) \le \mathscr{C}_{\alpha,q}(K)$ for every $\mu \in \mathfrak{M}^+(K)$ such that ${\mathcal{W}}_{\alpha,q}^{\mathcal{D}}\mu \le 1$ on $\mathrm{spt}(\mu)$. Indeed, Theorem \ref{Thm-capacitary-meas} guarantees the existence of a $\mathscr{C}$-capacitary measure $\mu$ for $K$ for which $\mathscr{C}_{\alpha,q}(K) = \mu(K)$. Therefore, assume $\mu\in\mathfrak{M}^+(K)$ such that ${\mathcal{W}}_{\alpha,q}^{\mathcal{D}}\mu\leq 1$ on spt($\mu$), and let $\gamma^{K}$ be a $\mathscr{C}-$capacitaty measure for $K$. Since $\gamma^K\ll \mathscr{C}_{\alpha,q}$, then ${\mathcal{W}}_{\alpha,q}^{\mathcal{D}}\gamma^K\geq 1$ for $\;\gamma-$a.e on $\;K$.  Now mimicking the proof of \eqref{Energy-Energy}, using ${\mathcal{E}}\mu\leq \mu(K)$ and \eqref{key-cap-dya-ener} we obtain  
\begin{align*}
	\mu(K)\leq \int {\mathcal{W}}_{\alpha,q}^{\mathcal{D}}\gamma^Kd\mu \leq \left({\mathcal{E}}\gamma^K\right)^{1/q}\left({\mathcal{E}}\mu\right)^{1/q'}
\leq\mathscr{C}_{\alpha,q}(K)^{1/q}\mu(K)^{1/q'},
\end{align*}
which leads to $\mu(K)\leq\mathscr{C}_{\alpha,q}(K)$, as required. 
\end{proof}

\subsubsection{Capacitary measure for arbitrary sets}
This section is devoted to establishing the existence of $\mathscr{C}$-capacitary measures for general sets $E$. The classical proof relies on Clarkson's inequalities in $L^p$ (see \cite{Wolff}). However, it is enough to invoke  Lemma \ref{simple-ineqs} below. For the reader's convenience, we sketch the proof. 

\begin{lemma} \label{simple-ineqs}For all $a,b\in \mathbb{R}$ we have 
    \begin{itemize}
        \item[(i)]$\left\vert \frac{a-b}{2}\right\vert ^p+\left\vert \frac{a+b}{2}\right\vert ^p\leq 2^{1-p}(\vert a\vert ^p+\vert b\vert ^p)$, as  $1<p\leq 2$. 
        \item[(ii)]
$\left\vert \frac{a-b}{2}\right\vert ^p+\left\vert \frac{a+b}{2}\right\vert ^p\leq \frac{1}{2}(\vert a\vert^p+\vert b\vert ^p)$, as $2\leq p<\infty$.
    \end{itemize}
\end{lemma}

\begin{proof} Since the function $\mathbb{R}_+\ni  t\mapsto t^{r/s}$ is convex as $s\leq r$, then $M(s)=\left(\sum_{i=1}^{n}\frac{x_i^s}{n}\right)^{1/s}$ is an  increasing function. Indeed, by Jensen's inequality we have $\left(\sum_{i=1}^{n}\frac{x_i^s}{n}\right)^{r/s}\leq \sum_{i=1}^{n}\frac{x_i^r}{n}$ which leads to $M(s)\leq M(r)$. Therefore, $M(p)\leq M(2)$ for all $1<p\leq 2$. In particular, for $x_1=\vert a-b\vert /2$ and $x_2=\vert a+b\vert /2$ we have 
    \begin{align}
        \frac{x_1^p+x_2^p}{2}\leq \left(\frac{x_1^2 +x_2^2}{2}\right)^{\frac{p}{2}}=\left(\frac{a^2 +b^2}{4}\right)^{\frac{p}{2}}\leq 2^{-p}(\vert a\vert ^p+\vert b\vert ^p),\nonumber
    \end{align}
    since $p/2\leq 1$. Note that $x_1^p+1\leq \left( x_1^2+1\right)^{p/2}$ as $2\leq p< \infty$, since the function $f(t)=(t^2+1)^{p/2}-t^p-1\geq 0$. Hence, we have   $x_1^p+x_2^p\leq \left( x_1^2+x_2^2\right)^{p/2}$, and by convexity of the function $t\mapsto t^{p/2}$, we get  
    \begin{align}
        x_1^p+x_2^p\leq \left( x_1^2+x_2^2\right)^{p/2} =\left(\frac{a^2+b^2}{2}\right)^{p/2}\leq \frac{1}{2}\vert a\vert^p +\frac{1}{2}\vert b\vert ^p,\nonumber
    \end{align}
	as desired. 
\end{proof}

\begin{proposition}\label{Prop:extremal-aberto}Let $\mathcal{O} \subset \mathbb{R}^{d+1}$ be such that  $\mathscr{C}_{\alpha,q}(\mathcal{O})< \infty$. 
    \begin{itemize}
        \item [\textnormal{\textbf{(A)}}] If $\mathcal{O}$ is an open set, then there exists a $\mathscr{C}-$extremal measure $\gamma$ for $\mathcal{O}$ such that 
		\begin{equation}\label{ineq: Prop-A1}
			({\mathcal{W}}_{\alpha,q}^{\mathcal{D}}\gamma)(x,t)\geq {\mathcal{E}}\gamma\quad\mathscr{C}_{\alpha,q}-\text{a.e}\quad  (x,t)\in \mathcal{O}.
		\end{equation}
        \item [\textnormal{\textbf{(B)}}] If $\mathcal{O}$ is an arbitrary set, then there exists a $\mathscr{C}-$extremal measure $\gamma$ for $\mathcal{O}$ such that \begin{equation}\label{ineq: Prop-A2}
			({\mathcal{W}}_{\alpha,q}^{\mathcal{D}}\gamma)(x,t)\geq {\mathcal{E}}\gamma\quad\mathscr{C}_{\alpha,q}-\text{a.e}\quad  (x,t)\in \mathcal{O}.
		\end{equation}
        \item [\textnormal{\textbf{(C)}}] $({\mathcal{W}}_{\alpha,q}^{\mathcal{D}}\gamma)(x,t)\leq  {\mathcal{E}}\gamma\;$  \text{ for all }  $(x,t)\in \textnormal{spt}(\gamma)$.
        \end{itemize}
\end{proposition}

\begin{proof} Since $\mathscr{C}_{\alpha,q}(\cdot)\simeq C_{\alpha,q}(\cdot)$ and open sets are $\mathscr{C}-$capacitable (see Proposition \ref{Prop-conv-capacity}), we can choose an increasing sequence  $K_n\subset K_{n+1}\subset\mathcal{O}$  of compact sets such that $K=\bigcup_nK_n$ and $\mathscr{C}_{\alpha,q}(K_n)\rightarrow  \sup_{K_n \subset \mathcal{O}}\mathscr{C}_{\alpha,q}(K)=\mathscr{C}_{\alpha,q}(\mathcal{O})$. By Lemma \ref{Cap-extremal-meas}, let $\{\gamma_n\}\in \mathfrak{M}^+_1(K_n)$ be a sequence of $\mathscr{C}-$extremal measures for $K_n$. By weak compactness we can find a subsequence $\{\gamma_{n_j}\}$ and a Radon measure $\gamma\in \mathfrak{M}^+(\overline{\mathcal{O}})$ such that $\gamma_{n_j}\rightharpoonup \gamma$. But $\mu\mapsto {\mathcal{E}}\mu$ is lower semicontinuous on weak$^{\star}-$topology of $\mathfrak{M}^+$, then 
	\begin{equation}\label{ineq:liminf-E}
		{\mathcal{E}}\gamma \leq \liminf_{j\rightarrow \infty}{\mathcal{E}}\gamma_{n_j}.
	\end{equation}
We claim that ${\mathcal{E}}\gamma_{n_j}\searrow [\mathscr{C}_{\alpha,q}(\mathcal{O})]^{1-q'}$ and ${\mathcal{E}}\gamma_{n_j}\leq {\mathcal{E}}\left((\gamma_{n_j}+\gamma_{m_j})/2\right)$. Indeed, since $\mathscr{C}_{\alpha,q}(K_{n_j})\leq \mathscr{C}_{\alpha,q}(K_{n_{j+1}})$ we have ${\mathcal{E}}\gamma_{n_j} 
\geq{\mathcal{E}}\gamma_{n_{j+1}}$. Moreover, 
$$\lim_{j\rightarrow \infty}{\mathcal{E}}\gamma_{n_j}=  \lim_{j\rightarrow \infty} [\mathscr{C}_{\alpha,q}(K_{n_j})]^{1-q'} = [\mathscr{C}_{\alpha,q}(\mathcal{O})]^{1-q'}.$$
The assertion ${\mathcal{E}}\gamma_{n_j}\leq {\mathcal{E}}\left((\gamma_{n_j}+\gamma_{m_j})/2\right)$ can be established by observing that $(\gamma_{n_j}+\gamma_{m_j})/2$ belongs to $\mathfrak{M}^+_1$, and ${\mathcal{E}}(\cdot)$ minimizes energy on $\mathfrak{M}^+_1$ (see Definition \ref{extremal-meas}). 

We claim that $\gamma$ is a $\mathscr{C}-$extremal measure for the open set $\mathcal{O}$. To this end, we will show that ${\mathcal{E}}\gamma =\lim_{j\rightarrow \infty}{\mathcal{E}}\gamma_{n_j}$. Indeed, recall that  ${\mathcal{E}}\gamma_{n_j}\searrow [\mathscr{C}_{\alpha,q}(\mathcal{O})]^{1-q'}$ and ${\mathcal{E}}\gamma_{m_j}\searrow [\mathscr{C}_{\alpha,q}(\mathcal{O})]^{1-q'}$ as $n_j, m_j \rightarrow \infty$. Then, by Lemma \ref{simple-ineqs} we have
\begin{align}
	\mathcal{E}_{n_j,m_j}&:=\mathcal{E}\left(\gamma_{n_j}-\gamma_{m_j}\right)= 2^{q'}\sum_{\ell_R<1} b(R)\left(\frac{\gamma_{n_j}(\eta_R)-\gamma_{m_j}(\eta_R)}{2}\right)^{q'}\nonumber\\ 
		&\leq 2^{q'}c_{q'}\left\{ \sum_{R} b(R)\gamma_{n_j}(\eta_R)^{q'} +\sum_{R} b(R)\gamma_{m_j}(\eta_R)^{q'} \right\}-2^{q'}\sum_{\ell_R<1} b(R)\left(\frac{\gamma_{n_j}(\eta_R)+\gamma_{m_j}(\eta_R)}{2}\right)^{q'}\nonumber\\
		&=2^{q'}c_{q'}\{{\mathcal{E}}\gamma_{n_j}+{\mathcal{E}}\gamma_{m_j}\}-2^{q'} {\mathcal{E}}\left(\frac{\gamma_{n_j}+\gamma_{m_j}}{2}\right)\nonumber\\
		&\leq 2^{q'}c_{q'}\{{\mathcal{E}}\gamma_{n_j}+{\mathcal{E}}\gamma_{m_j}\}-2^{q'} {\mathcal{E}}\gamma_{n_j} \;\rightarrow\; 0,\label{ineq:conv-cauchy}
\end{align}
where $c_{q'}=(1/2)^{q'-1}$ for $1<q'\leq 2$ or  $c_{q'}=1/2$ for $2<q'<\infty$. Therefore, for every $\varepsilon>0$ and $m_j$ large enough we have 
\begin{equation}
{\mathcal{E}}(\gamma_{m_j}-\gamma)\leq \liminf_{j\rightarrow \infty} {\mathcal{E}}(\gamma_{m_j}-\gamma_{n_j}) <\varepsilon. \label{ineq:conv-energy}
\end{equation}
Now we are ready to prove \eqref{ineq: Prop-A1}. Let $E_{\vartheta} = \{(x,t)\in \mathcal{O}\,:\,[{\mathcal{W}}_{\alpha,q}^{\mathcal{D}}(\gamma_{m_j}-\gamma)](x,t)> \vartheta\}$. Note that Theorem  \ref{Thm-carac-capacity}(A) and  \eqref{ineq:conv-energy} imply  $\mathscr{C}_{\alpha,q}(E_{\vartheta})\leq \vartheta^{-q} {\mathcal{E}}(\gamma_{m_j}-\gamma):=\vartheta^{-q}{\mathcal{E}}_{m_j}(\gamma)\rightarrow 0$ as $m_j\rightarrow \infty$. Now let $\varepsilon>0$ and choose $\sum_{j=1}^{\infty} \vartheta_j^{-q}{\mathcal{E}}_{m_j}(\gamma) <\varepsilon $ as $\vartheta_j\rightarrow 0$ and $m_j\rightarrow \infty$.  Then $A=\bigcup_j E_{\vartheta_j}$ has null capacity, since $\mathscr{C}_{\alpha,q}(A)\leq \sum_{j=1}^{\infty} \vartheta_j^{-q}{\mathcal{E}}_{m_j}(\gamma) <\varepsilon$. Therefore, for all $(x,t)\in\mathcal{O}\backslash A$  we obtain 
\begin{equation}
	[{\mathcal{W}}_{\alpha,q}^{\mathcal{D}}(\gamma_{m_j}-\gamma)](x,t)\leq \vartheta_j.
\end{equation}
This shows that ${\mathcal{W}}_{\alpha,q}^{\mathcal{D}}\gamma_{m_j}$ converges to ${\mathcal{W}}_{\alpha,q}^{\mathcal{D}}\gamma$ $\;\;\;\mathscr{C}-$a.e on $\mathcal{O}$. Now, invoking Proposition \ref{prop-equilibrium} and inequality \eqref{ineq:conv-energy}, we obtain the desired result as required.  

Let us proof \textbf{(B)}. Again, by Proposition \ref{Prop-conv-capacity}, we can choose a decreasing sequence $\overline{A_{n+1}}\subset A_{n}$ of open sets $A_n\supseteq \mathcal{O}$ such that $\overline{\mathcal{O}}=\bigcap _{n=1}^{\infty}\overline{A_n}$ and $\mathscr{C}_{\alpha,q}(A_n)\rightarrow  \mathscr{C}_{\alpha,q}(\mathcal{O})$ as $n\rightarrow \infty$. By previous item, let $\gamma_n\in\mathfrak{M}^{+}(A_n)$ be a $\mathscr{C}-$extremal measure for $A_n$. Following the proof of \textbf{(A)} and using Lemma \ref{simple-ineqs}, there exists a $\gamma \in \mathfrak{M}^{+}(\overline{\mathcal{O}})$ such that ${\mathcal{E}}(\gamma_{m_j} - \gamma) \rightarrow 0$ as $m_j \rightarrow \infty$. Hence, by  previous argument,  we obtain \eqref{ineq: Prop-A2} as required. 

\end{proof}  
Let $E \subset \mathbb{R}^{d+1}$ be an arbitrary set such that $0<\mathscr{C}_{\alpha,q}(E)<\infty$. Then, in view of Proposition \ref{Prop:extremal-aberto}, and  arguing as in the proof of Theorem \ref{Thm-capacitary-meas}, there exists a $\mathscr{C}$-capacitary measure for $E$.

\begin{corollary}\label{cor-capacitary-meas} Let $E\subset \mathbb{R}^{d+1}$ be an arbitrary set such that $0<\mathscr{C}_{\alpha,q}(E)<\infty$. Then $E$ admits a $\mathscr{C}_{\alpha,q}-$capacitary measure $\gamma^E$. 
    \end{corollary}

Adapting the proof of Theorem \ref{Thm-carac-capacity}, this theorem extends naturally to arbitrary sets.

\begin{theorem}\label{Cor-carac-capacity}\ 
	\begin{itemize}
	\item [\textnormal{\textbf{(A)}}]  If $\mu\in\mathfrak{M}^+$  has finite energy  ${\mathcal{E}}\mu<\infty$, then for every arbitrary set $E\subset E_{\vartheta}=\{ (x,t)\,:\, ({\mathcal{W}}_{\alpha,q}^{\mathcal{D}}\mu)(x,t)>\vartheta\}$ we have 
		\begin{equation}
			\mathscr{C}_{\alpha,q}(E)\leq \vartheta^{-q}{\mathcal{E}}\mu, \quad \vartheta>0.\nonumber
		\end{equation}
	\item [\textnormal{\textbf{(B)}}] Let $E\subset\mathbb{R}^{d+1}$ be an arbitrary set, then 
\begin{equation}
			\mathscr{C}_{\alpha,q}(E)=\inf\left\{ {\mathcal{E}}\mu\,:\; \mu\in\mathfrak{M}^+ \text{ and }  {\mathcal{W}}_{\alpha,q}^{\mathcal{D}}\mu\geq 1\;\;\;  \mathscr{C}_{\alpha,q}-\text{a.e on }E\right\}\nonumber.
	\end{equation}
	\item [\textnormal{\textbf{(C)}}] Let $E\subset\mathbb{R}^{d+1}$ be an arbitrary set, then 
\begin{equation}
	\mathscr{C}_{\alpha,q}(E)= \sup \left\{ \mu(E) \,:\; \mu\in\mathfrak{M}^+(E) \text{ and } {\mathcal{W}}_{\alpha,q}^{\mathcal{D}}\mu\leq 1\;\text{ on }\;\textnormal{spt}(\mu)\right\}\nonumber.
\end{equation}
	\end{itemize}
\end{theorem}

\section{Thermal thinness, Kellogg and Choquet properties}\label{chapter-five}
In this section, we establish the existence of an open set $\mathcal{O}\subset \mathbb{R}^{d+1}$ containing $e^{\mathscr{P}}_{\alpha,q}(E)$ such that
\[
\mathscr{C}_{\alpha,q}(E\cap \mathcal{O})<\varepsilon \quad \text{for all} \quad \varepsilon>0,
\]
called parabolic Choquet property. Moreover, we prove the two Kellogg properties and Theorem \ref{Conjuntos finos Geral parabolico1}.

 \subsection{Thermal thinness}\label{thermal-thinness}  
Following the approaches of \cite[Proposition 1.2]{FGL} and \cite[Theorem 1.13]{Brzezina}, one obtains several equivalent characterizations of parabolic $(\alpha,q)$-thinness. 
\begin{proposition}\label{equiv-thin} The following statements are equivalents.
	\begin{itemize}
		\item[(i)]  For each $z\in \overline{E}$ we have  
			\begin{equation}
			W_{\alpha,q}^E(z)=\int_{0}^{1}\left(\frac{\mathscr{C}_{\alpha,q}(E \cap Q_{r}(z))}{r^{n-\alpha q}}\right)^{q'-1}\frac{dr}{r}< \infty.
		\end{equation}
		\item[(ii)]  Let $Q_{j}(z)=Q_{2^{-j}}(z)$, then 
			\begin{align}
			\sum_{j=1}^{\infty} \left(2^{j(n-\alpha q)}\mathscr{C}_{\alpha,q}(E\cap Q_{j}(z))\right)^{q'-1}<\infty.
		\end{align}
		\item[(iii)] Let $A_j(z)=\overline{Q_j(z)\backslash Q_{j+1}(z)}$ be the close of the annulus $Q_j(z)\backslash Q_{j+1}(z)$, then  
			\begin{align}
			\sum_{j=1}^{\infty} \left(2^{j(n-\alpha q)}\mathscr{C}_{\alpha,q}(E\cap A_{j}(z))\right)^{q'-1}<\infty.
		\end{align}
	\end{itemize}

\end{proposition}

\begin{proof} Since $\mathscr{C}_{\alpha,q}(E\cap Q_{2^{-j}}(z))\lesssim 2^{-j(n-\alpha q)}$ (see Proposition \ref{basic-properties-cap}(v)), then $(i)$ is equivalent to $(ii)$. Indeed, we have 
	\begin{align}
	W^E_{\alpha,q}(z)=\sum_{j=0}^{\infty}\int_{2^{-j-1}}^{2^{-j}}\Big(\frac{\mathscr{C}_{\alpha,q}(E\cap Q_{r}(z))}{r^{n-\alpha q} }\Big)^{q'-1}\frac{dr}{r}&\simeq \sum_{j=0}^{\infty} \left(2^{j(n-\alpha q)}\mathscr{C}_{\alpha,q}(E\cap Q_{j}(z))\right)^{q'-1}\nonumber\\
		&= \sum_{j=1}^{\infty} \left(2^{j(n-\alpha q)}\mathscr{C}_{\alpha,q}(E\cap Q_{j}(z))\right)^{q'-1}+c\nonumber,
\end{align}
where $c=\mathscr{C}_{\alpha,q}(E\cap Q_1(z))^{q'-1}\lesssim1$. Clearly $(ii)$ implies $(iii)$, since $ E\cap A_j(z) \subset E\cap Q_j(z)$. It remains to verify  that  $(ii)$ can be obtained from $(iii)$. To this end, 
since  $Q_j(z)\subset A_j(z)\cup Q_{j+1}(z)$ we obtain 
\begin{equation}
\mathscr{C}_{\alpha,q}(E\cap Q_{j}(z))\leq \mathscr{C}_{\alpha,q}(E\cap A_{j}(z))\cup \mathscr{C}_{\alpha,q}(E\cap Q_{j+1}(z))\nonumber,
\end{equation}
which leads to 
\begin{align}
	\sum_{j=1}^{k} b_j^{q'-1}\leq \sum_{j=1}^{k}a_j^{q'-1}
	+\sum_{j=1}^{k-1} \left(2^{j(n-\alpha q)} \mathscr{C}_{\alpha,q}(E\cap Q_{j+1})\right)^{q'-1}+\left(2^{k(n-\alpha q)}\mathscr{C}_{\alpha,q}(E\cap Q_{k+1})\right)^{q'-1}\nonumber
\end{align}
where $b_j(z)=2^{j(n-\alpha q)}\mathscr{C}_{\alpha,q}(E\cap Q_{j}(z))$ and $a_j(z)=2^{j(n-\alpha q)}\mathscr{C}_{\alpha,q}(E\cap A_{j}(z))$. Now by a tedious computation, we can write 
\begin{align}
	\sum_{j=1}^{k} b_j^{q'-1}-\sum_{j=1}^{k-1} \left(2^{j(n-\alpha q)} \mathscr{C}_{\alpha,q}(E\cap Q_{j+1})\right)^{q'-1}&=  \sum_{j=1}^{k} b_j^{q'-1}-2^{-(n-\alpha q)(q'-1)}\sum_{j=1}^{k-1} b_{j+1}^{q'-1}\nonumber\\
	&= c \sum_{j=2}^{k} b_j^{q'-1}+ \left(2^{n-\alpha q}\mathscr{C}_{\alpha,q}(E\cap Q_{1}(z))\right)^{q'-1}\nonumber\\
		&= \mathscr{C}_{\alpha,q}(E\cap Q_{1}(z))^{q'-1}+c \sum_{j=1}^{k} b_j^{q'-1},\nonumber
\end{align}
where $c=1- 2^{-(n-\alpha q)(q'-1)}$. By Proposition \ref{basic-properties-cap}(v), we have $2^{k(n-\alpha q)}\mathscr{C}_{\alpha,q}(E\cap Q_{k}(z))\lesssim 1$. Then, 
\begin{align}
\sum_{j=1}^{\infty} b_j(z)^{q'-1}&\lesssim \sum_{j=1}^{\infty} \left(2^{j(n-\alpha q)}\mathscr{C}_{\alpha,q}(E\cap A_{j}(z))\right)^{q'-1} + 1 <\infty \nonumber
\end{align}
as desired.
\end{proof}

\begin{lemma}[$\mathscr{C}-$quasicontinuously] \label{quasicont}Let $\mu\in\mathfrak{M}^+$ be such that $\mathcal{E}\mu <\infty$. Then, for every $\epsilon>0$, there exist an open set $\mathcal{O}\subset \mathbb{R}^{d+1}$ with $\mathscr{C}_{\alpha,q}(\mathcal{O})<\epsilon$ such that $\mathcal{W}_{\alpha,q}^{\mathcal{D}}\mu$ is continuous on ${\mathcal{O}}^c$.
\end{lemma}

\begin{proof} For each $N\in\mathbb{N}$, let us define 
    $$({\mathcal{W}}^{\mathcal{D}}_N\mu)(x,t)=\sum_{j=1}^{N}\sum_{\;\;\ell_R=2^{-i}\ell_j}(\ell_R^{\alpha q-n}\mu(\eta_R))^{q'-1}\eta_R(x,t).$$
Then ${\mathcal{W}}^{\mathcal{D}}_N\mu$ is continuous in $\mathbb{R}^{d+1}$. Furthermore, since ${\mathcal{E}}\mu$ is finite, we have 
\begin{equation}
    {\mathcal{E}}_N(\mu)=\int ({\mathcal{W}}_{\alpha,q}^{\mathcal{D}}\mu -{\mathcal{W}}^{\mathcal{D}}_N\mu)d\mu = \sum_{j>N}\sum_{\;\;\ell_R=2^{-i}\ell_j}b(R) \mu(\eta_R)^{q'}\rightarrow 0 \quad \text{ as } \quad N\rightarrow \infty. \nonumber
\end{equation}
Now consider the set  $E_{\vartheta}= \left\{ (x,t)\,:\, {\mathcal{W}}_{\alpha,q}^{\mathcal{D}}\mu(x,t)-{\mathcal{W}}_N^{\mathcal{D}}\mu(x,t)> \vartheta\right\}$ for $\vartheta>0$. Following the same steps of Theorem  \ref{Thm-carac-capacity}(A), we have 
\begin{align}
\vartheta \leq  \int ({\mathcal{W}}_{\alpha,q}^{\mathcal{D}}\mu -{\mathcal{W}}^{\mathcal{D}}_N\mu)d\gamma &= \sum_{j>N}\sum_{\;\ell_R=2^{-i}\ell_j} (\ell_R^{\alpha q-n}\mu(\eta_R))^{q'-1}\gamma(\eta_R)\nonumber\\
&\leq (\mathcal{E}_N(\mu))^{1/q} (\mathcal{E}_N(\gamma))^{1/q'},\nonumber
\end{align}
where $\gamma$ is a $\mathscr{C}-$extremal measure for compact set $K\subset E_{\vartheta}$. Then  $\mathscr{C}_{\alpha,q}(K)\leq \vartheta^{-q}{\mathcal{E}}_N(\mu)$
which yields $\mathscr{C}_{\alpha,q}(E_{\vartheta})\leq \vartheta^{-q}{\mathcal{E}}_N(\mu)$. {Now,  for every  $\varepsilon>0$, choose $\sum_{j=1}^{\infty} \vartheta_j^{-q}{\mathcal{E}}_{N_j}(\mu) <\varepsilon $ as $\vartheta_j\rightarrow 0$ and $N_j\rightarrow \infty$.  Then the open set $\mathcal{O}=\bigcup_{j}E_{\vartheta_j}$ has null capacity, since $\mathscr{C}_{\alpha,q}(\mathcal{O})\leq \sum_{j=1}^{\infty} \vartheta_j^{-q}{\mathcal{E}}_{N_j}(\gamma) <\varepsilon$. Therefore, for  $(x,t)\in \mathcal{O}^c$  we obtain 
\begin{equation}
	{\mathcal{W}}_{\alpha,q}^{\mathcal{D}}\mu(x,t) -{\mathcal{W}}^{\mathcal{D}}_{N_j}\mu(x,t)\leq \vartheta_j.
\end{equation}
Hence, the sequence ${\mathcal{W}}_{N_j}^{\mathcal{D}}\mu $ converges (uniformly) to ${\mathcal{W}}_{\alpha,q}^{\mathcal{D}}\mu$ for $\mathscr{C}-$a.e on $\mathcal{O}$.} In particular, ${\mathcal{W}}_{\alpha,q}^{\mathcal{D}}\mu$ is continuous on $\mathcal{O}^c$.
\end{proof}

\subsection{Kellogg property} 
In this section we are interested in the parabolic Kellogg property. To this end, the next lemma provides a key ingredient. 
\begin{lemma}\label{lemao}Assume $0<\alpha<n$ and $1<q\leq n/\alpha$. 
	\begin{itemize}		
		\item[\textnormal{\textbf{(i)}}]	Let $E\subset \mathbb{R}^{d+1}$ satisfy  $0<\mathscr{C}_{\alpha,q}(E)< \infty$, and let $\mu$ denote a $\mathscr{C}_{\alpha,q}-$capacitary measure for $E$. Then, for every Borel set $S\subset \mathbb{R}^{d+1}$,  
		\begin{equation}
			\mu(S) \leq \mathscr{C}_{\alpha,q}(E \cap S).\nonumber
		\end{equation}
		
		\item[\textnormal{\textbf{(ii)}}] Assume that $E$ is parabolic $(\alpha,q)-$thin at $z_0$. Then, for every $\varepsilon>0$, there exists $N\in\mathbb{N}$ such that for every set $V\subset Q_{2^{-N}}(z_0)$ one can find a $\mathscr{C}_{\alpha,q}-$capacitary measure $\mu$ for $E\cap V$ satisfying  
		\begin{equation}
			{\mathcal{W}}_{\alpha,q}^{\mathcal{D}}\mu (z_0)<\varepsilon,\qquad {\mathcal{W}}_{\alpha,q}^{\mathcal{D}}\mu\geq 1 \;\text{ for }\; {\mathscr{C}}-a.e.\,\text{  on  }\;E \cap V.\nonumber
		\end{equation}
	\end{itemize}
\end{lemma}

\begin{proof}The proof of \textbf{(i)} follows by method outlined in \cite[Proposition 6.3.13]{AH}, substituting $\mathcal{V}_{\alpha,p}^{\mu}$ by ${\mathcal{W}}_{\alpha,q}^{\mathcal{D}}\mu$. Alternatively, we can apply the argument of \cite[Proposition 9]{Wolff}. In this case, the Theorem \ref{Thm-capacitary-meas} must be employed to establish the existence of a sequence of $\mathscr{C}_{\alpha,q}-$capacitary measures $\mu_n$ for  compact sets $F_n$. Then by Theorem \ref{Thm-carac-capacity}(\textbf{C}), we obtain $\mu_n(K) \leq \mathscr{C}_{\alpha,q}(K \cap F_n)$. 

Let us proof \textbf{(ii)}. First, by Corollary \ref{cor-capacitary-meas} we can find a  $\mathscr{C}-$capacitary measure $\mu=\mu^{E\cap V}$ for $E\cap V$. Recall that $\mathds{1}_R\leq \eta_R\leq \mathds{1}_{\delta_3R}$ and $\mu(\eta_R)=\int \eta_Rd\mu$ for every dyadic rectangle $R$.  Then, 
\begin{align}\label{5-control1}
	{\mathcal{W}}^{\mathcal{D}}_{\alpha,q}\mu(z_0)&\lesssim \sum_{\ell_R<1}\Big(\dfrac{\mu(\delta_3R)}{\ell_{R}^{n-\alpha q}}\Big)^{q'-1}\mathds{1}_{R}(z_0)\nonumber\\
		&\leq \sum_{j=1}^{\infty} \left(2^{-j(\alpha q-n)}\mu(Q_{3\sqrt{d}\, 2^{-j}}(z_0))\right)^{q'-1},
\end{align}
since for all dyadic time-backward rectangle $R\in\mathscr{D}_j$ containing $z_0$, we obtain $\delta_3R\subset Q_{3\sqrt{d}\,2^{-j}}(z_0)$. Without loss of generality, we assume $3\sqrt{d}=1$. Since $E$ is parabolic $(\alpha,q)-$thin at $z_0$, then \eqref{5-control1},  Lemma \ref{lemao}\textbf{(i)} and Proposition \ref{equiv-thin} leads to 
\begin{align}
	\mathcal{W}_{\alpha,q}^{\mathcal{D}}\mu(z_0)&\lesssim \sum_{j\in\mathbb{N}} \left(2^{j(n-\alpha q)}\mathscr{C}_{\alpha,q}(E\cap Q_{2^{-j}}(z_0))\right)^{q'-1}<\infty \label{conv-thin}.
\end{align}
 Therefore, for any $\varepsilon'> 0$, there exist $N\in\mathbb{N}$ sufficiently large such that
\begin{align}
	\sum_{j=N}^{\infty} \left(2^{j(n-\alpha q)}\mathscr{C}_{\alpha,q}(E\cap Q_{2^{-j}}(z_0))\right)^{q'-1}<\varepsilon'. \nonumber
\end{align}
Once more,  by Lemma \ref{lemao}\textbf{(i)} with $\mu=\mu^{E\cap V}$ where $V\subset Q_{2^{-N}}(z_0)$, we can derive 
\begin{equation}
	\mu(Q_{2^{-j}}(z_0)) \leq \mathscr{C}_{\alpha,q}(Q_{2^{-j}}(z_0)\cap E\cap V)
\leq \mathscr{C}_{\alpha,q}(E\cap V)\leq \mathscr{C}_{\alpha,q}(E\cap Q_{2^{-N}}(z_0)).\nonumber
\end{equation}
Moreover, by convergence \eqref{conv-thin} we obtain 
\begin{equation}
	\lim_{N\rightarrow \infty} \left(2^{N(n-\alpha q)}\mathscr{C}_{\alpha,q}(E\cap Q_{2^{-N}}(z_0))\right)^{q'-1} =0.\nonumber
\end{equation}
Then, choosing  $\varepsilon'=\varepsilon/2$  and $N$ sufficiently large we have 
\begin{align}
	\mathcal{W}_{\alpha,q}^{\mathcal{D}}\mu(z_0)&\lesssim \varepsilon' + \sum_{j=1}^{N-1} \left(2^{j(n-\alpha q)}\mathscr{C}_{\alpha,q}(E\cap Q_j(z_0))\right)^{q'-1}\nonumber\\
	&\leq \varepsilon'+\mathscr{C}_{\alpha,q}(E\cap Q_{2^{-N}}(z_0))^{q'-1}\sum_{j=1}^{N-1}2^{-j(n-\alpha q)(q'-1)}\nonumber\\
	&=\varepsilon'+ \big( 2^{N(n-\alpha q)}\mathscr{C}_{\alpha,q}(E\cap Q_{2^{-N}}(z_0)\big)^{q'-1}\sum_{j=1}^{N-1}2^{(j-N)(n-\alpha q)(q'-1)}\nonumber\\
	&\leq \varepsilon'+ \varepsilon' \sum_{j=1}^{\infty} 2^{(j-N)(n-\alpha q)(q'-1)}\quad \text{ for }\quad  j<N\nonumber\\
	&\lesssim \varepsilon, \nonumber
\end{align}
this finish the proof.  
\end{proof}

\begin{proposition}[Parabolic Kellogg property]\label{Kellogg} Let  $0<\alpha<n$ and $1<q\leq n/\alpha$. Then, for every $E\subset\mathbb{R}^{d+1}$, we have 
	\begin{equation}
		\mathscr{C}_{\alpha,q}\big(E\cap e^{\mathscr{P}}_{\alpha,q}(E)\big)=0.\nonumber
	\end{equation}
\end{proposition}

\begin{proof}
Let $E\subset \mathbb{R}^{d+1}$ and $\{\mathcal{O}_n\}_{n=1}^{\infty}$ be an enumeration of parabolic open balls in $\mathbb{R}^{d+1}$ with radius sufficiently small such that $E\cap \mathcal{O}_n\neq \emptyset$. Let $\{\mu_n\}$ be a sequence of $\mathscr{C}-$capacitary measure for $E\cap\mathcal{O}_n$ and set $A_n=\{z\in \overline{E}\cap\mathcal{O}_n\,:\, ({\mathcal{W}}_{\alpha,q}^{\mathcal{D}}\mu_n)(z)<1\}$. Then, by Lemma \ref{lemao}(ii), we have $\overline{E}\cap e_{\alpha,q}^{\mathscr{P}}(E)\subset \bigcup_{n=1}^{\infty} (E\cap A_n)$. But  ${\mathcal{W}}_{\alpha,q}^{\mathcal{D}}\mu_n\geq 1$ $\;\mathscr{C}_{\alpha,q}-$a.e \,on $\;E\cap \mathcal{O}_n$, then $\mathscr{C}_{\alpha,q}(E\cap A_n)=0$ for all $n$. It follows that 
$
	\mathscr{C}_{\alpha,q}\left(E\cap e^{\mathscr{P}}_{\alpha,q}(E)\right)\leq \sum_{n=1}^{\infty} \mathscr{C}_{\alpha,q}(E\cap A_n) =0,\nonumber
$
confirming the result. 
\end{proof}

To treat the heal-ball formulation $C_{\alpha,2}\big(E\cap e^{\text{heat}}_{\alpha,2}(E)\big)=0$, we require the following lemma. 
\begin{lemma}\label{olema-heat} Let $0<\alpha<n/2$ and $E\subset \mathbb{R}^{d+1}$. Suppose that $z_0\in \overline E$ satisfies 
	\begin{equation}\label{cond-heat-thin}
	\int_{0}^{1} \left( \frac{\dot{C}_{\alpha,2}(E \cap \varTheta^{2\alpha}_r(z_0))}{r^{{(n-2\alpha)}/{2}}} \right) \frac{dr}{r} < \infty.
	\end{equation}
	Then for every $\varepsilon>0$, there exists $R>0$ and $\dot{C}_{\alpha,2}$-capacitary measure $\mu^{A}$ for $A=E\cap \overline {Q_R(z_0)}$ such that 
	\begin{equation}
		\Gamma^{2\alpha}\ast\mu^A(z_0)<\varepsilon, \qquad \Gamma^{2\alpha}\ast\mu^A(z)\geq 1 \quad \dot{C}_{\alpha,2}\text{-a.e.}\quad z\in A \nonumber.
	\end{equation}
\end{lemma}

\begin{proof} By Proposition \ref{prop-aikawa}, adapted to $\dot C_{\alpha,2}$ (see Corollary \ref{Cap-and-Cap_T}(i)), there exist $R>0$ and a $\dot C_{\alpha,2}$-capacitary measure $\mu^A$ supported on $A\subset \overline{Q_R(z_0)}$ such that 
\begin{equation}
	\mu^A(A)=\dot C_{\alpha,2}(A)\quad \text{and}\quad 	\Gamma^{2\alpha}\ast \mu^A\geq 1 \quad \dot C_{\alpha,2}\text{-a.e on } A.
\end{equation}	
By Lemma \ref{Riesz-heat-ball}, Lemma \ref{lemao}(i), and monotonicity of capacity,  
\begin{align}
	\Gamma^{2\alpha}\ast \mu^A(z_0)\leq \int_0^R \frac{\mu^A(\varTheta^{2\alpha}_r(z_0))}{r^{(n-2\alpha)/2}}\frac{dr}{r}
&\leq \int_0^R \frac{\dot C_{\alpha,2}(E\cap \varTheta^{2\alpha}_r(z_0))}{r^{(n-2\alpha)/2}}\frac{dr}{r}\nonumber.
\end{align}
Then for every $\varepsilon>0$, the condition \eqref{cond-heat-thin}  implies $\Gamma^{2\alpha}\ast \mu^A(z_0)<\varepsilon$ for sufficiently small $R>0$. 
\end{proof}
 
\begin{proposition}\label{Kellogg-heat} Let  $0<\alpha<n/2$. Then, for every $E\subset\mathbb{R}^{d+1}$, we have 
	\begin{equation}
		{C}_{\alpha,2}\big(E\cap e^{\textnormal{heat}}_{\alpha,2}(E)\big)=0.\nonumber
	\end{equation}
\end{proposition}

\begin{proof}The proof follows the argument of Proposition~\ref{Kellogg}, with $\mathscr C_{\alpha,q}$ replaced by $\dot C_{\alpha,2}$ and the dyadic Wolff potential $\mathscr W_{\alpha,q}$ replaced by the heat potential $\Gamma^{2\alpha}$. Moreover, the Lemma~\ref{lemao}(ii) is replaced by Lemma \ref{olema-heat}. We omit the details.
\end{proof}
	
\subsection{Choquet property}
In this section we establish the Choquet property.
\begin{theorem}[Parabolic Choquet property]\label{Thm:Choquet}
Let $0<\alpha<n$  and $1<q\leq n/\alpha$. For any set $E \subset \mathbb{R}^{d+1}$ and $\varepsilon>0$,  there exists an open set $G \subset \mathbb{R}^{d+1}$ such that 
	\begin{equation}\label{ineq:Choquet}
		e^{\mathscr{P}}_{\alpha,q}(E) \subset G \;\;\text{ and  }\;\; \mathscr{C}_{\alpha,q}(E \cap G)< \varepsilon.
	\end{equation}
\end{theorem}	
\begin{proof} Consider the sets $\mathcal{O}_n$ and $A_n$ as described in the proof of the parabolic Kellogg property. Let $\varepsilon>0$. Since $({\mathcal{W}}_{\alpha,q}^{\mathcal{D}}\mu_n)(x,t)\geq 1\;$ $\,\mathscr{C}_{\alpha,q}-$a.e on $\,E\cap \mathcal{O}_n\,$ and $\,{\mathcal{W}}_{\alpha,q}^{\mathcal{D}}\mu_n\,$ is $\mathscr{C}-$quasicontinuous (see Lemma \ref{quasicont}), then there is an open set $V_n$ with $\mathscr{C}_{\alpha,q}(V_n)<\varepsilon 2^{-n}$ such that ${\mathcal{W}}_{\alpha,q}^{\mathcal{D}}\mu_n$ is continuous out $V_n$. Then $({\mathcal{W}}_{\alpha,q}^{\mathcal{D}}\mu_n)(x,t)\geq 1\;$ for all $(x,t)\in E\cap\mathcal{O}_n\cap V_n^c$. Let us define the set $F=E\backslash \left(\bigcup_{n\geq1}V_n\right)=E\cap \left(\bigcap_{n\geq 1}V_n^c \right)$. Note that  $\overline{F}=\overline{E}\cap \left(\bigcap_{n\geq 1}V_n^c \right)$ and  $\overline{F}\cap \left(\bigcup_{n\geq 1}V_n\right)=\emptyset$. We claim that $G=\overline{F}^c$ satisfies \eqref{ineq:Choquet}. Indeed, by Lemma \ref{lemao}(ii), we have 
	\begin{equation}\label{Choquet:inclusion1}
	e_{\alpha,q}^{\mathscr{P}}(E)\subset \overline{E}^c\cap \Big(\bigcup_{n\geq 1} A_n\Big).
	\end{equation}
	By continuity of ${\mathcal{W}}_{\alpha,q}^{\mathcal{D}}\mu_n$ on $\overline{F}$ and $({\mathcal{W}}_{\alpha,q}^{\mathcal{D}}\mu_n)(x,t)\geq 1\;$ $\,\mathscr{C}_{\alpha,q}-$a.e \; on $\;\overline{F}\cap \mathcal{O}_n\,$, we obtain $({\mathcal{W}}^{\mathcal{D}}_{\alpha,q}\mu_n)(x,t)\geq 1\;$ for all $(x,t)\in \overline{F}\cap \mathcal{O}_n$. It follows that $\overline{F}\cap A_n=\emptyset$. In particular, $\left(\bigcup_{n\geq 1}A_n\right)\subset \overline{F}^c$, and  from inclusion \eqref{Choquet:inclusion1} we have 
	\begin{equation}
		e_{\alpha,q}^{\mathscr{P}}(E)\,\subset\, \overline{E}^c\cap \overline{F}^c\,\subset \,\overline{F}^c\nonumber.
		\end{equation}
		Moreover, we obtain $\mathscr{C}_{\alpha,q}(E\cap \overline{F}^c)\leq \mathscr{C}_{\alpha,q}\left(\bigcup_{n=1}^{\infty}V_n\right)\leq \sum_{n=1}^{\infty}\mathscr{C}_{\alpha,q}(V_n)<\varepsilon $.
\end{proof}

\subsection{Nonlinear parabolic thinness}\label{non-parab-thin}
Assume that $\mu\in\mathfrak{M}^+(\mathbb{R}^{d+1})$ has finite energy, from Corollary \ref{Cor-carac-capacity}(A) we have  
\begin{equation}\label{thm-non-finite-energy}
	\mathscr{C}_{\alpha,q}(\Big\{{\mathcal{W}}_{\alpha,q}^{\mathcal{D}}\mu(x,t)>\vartheta\Big\}) \leq  c\, \vartheta^{-q}\mathcal{E}\mu.
\end{equation}
However, the energy of $\mu$ is not finite in general. An alternative is presented below.

\begin{proposition} \label{infinity-energy}Let $\mu\in\mathfrak{M}^+(\mathbb{R}^{d+1})$ be a finite measure and $K$ be a compact subset of $E_{\vartheta}=\left\{{\mathcal{W}}_{\alpha,q}^{\mathcal{D}}\mu(x,t)>\vartheta\right\}$. If $0<\alpha<n$ and  $1<q\leq n/\alpha$, then there is a constant $c>0$ such that 
	\begin{equation}\label{thm-non-finite-energy2}
		\mathscr{C}_{\alpha,q}(K) \leq  c\, \vartheta^{1-q}\mu(\mathbb{R}^{d+1}),
	\end{equation}
	for all $\vartheta>0$.
\end{proposition}

\begin{proof}There exists a constant $c>0$ such that
    \[
	\textbf{W}_{\alpha,q}\mu(x,t) \le c\, \mathcal{W}_{\alpha,q}^{\mathcal{D}}\mu(x,t) \quad \text{for all } (x,t).
    \]
    Indeed, since $\mathds{1}_R \le \eta_R \le \mathds{1}_{\delta_3 R}$ and $\mu(\eta_R) = \int \eta_R \, d\mu$, the inequality \eqref{eq-Lemma3.8} implies
    \[
    \textbf{W}_{\alpha,q}\mu(x,t) \le c \sum_{\ell_R<1} \bigl(\ell_R^{\alpha q - n} \mu(R)\bigr)^{q'-1} \mathds{1}_R(x,t) \le c\, \mathcal{W}_{\alpha,q}^{\mathcal{D}}\mu(x,t).
    \]
Hence, it is sufficient to demonstrate the result for $\{{\textbf{W}}_{\alpha,q}\mu(x,t)>\vartheta\}$. Let $\mu\in\mathfrak{M}^+(\mathbb{R}^{d+1})$ and let $\gamma=\gamma^K\in\mathfrak{M}^+(\mathbb{R}^{d+1})$ be a $\mathscr{C}-$capacitary measure for $K\subset \{{\textbf{W}}_{\alpha,q}\mu(x,t)>\vartheta\}$, we define the centered maximal function $M_{\gamma}\mu$ as follows 
\begin{equation}
	(M_{\gamma}\mu)(x,t)=\sup\frac{\mu(Q_{5r}(x,t))}{\gamma(Q_r(x,t))},\nonumber
\end{equation}
where the supremum is taken over all backward parabolic ball $Q_r(x,t)$. Then, for all $(x,t)\in \text{spt}(\gamma)\subset \{{\textbf{W}}_{\alpha,q}\mu(x,t)>\vartheta\}$ we obtain
	\begin{align}
		\vartheta <{\textbf{W}}_{\alpha,q} \mu(x,t) = \int_0^{1} \left( \frac{\mu(Q_r(x,t))}{r^{n-\alpha q}} \right)^{q'-1} \frac{dr}{r}&\leq \left(\sup_{r>0}\frac{\mu(Q_{5r}(x,t))}{\gamma(Q_r(x,t))}\right)^{q'-1}{\textbf{W}}_{\alpha,q} \gamma (x,t) \nonumber\\
		  &\leq c [M_{\gamma}\mu(x,t)]^{q'-1},\nonumber
	\end{align}
since by Definition \ref{capacitary-meas-Wolff} one has ${\mathcal{W}}^{\mathcal{D}}_{\alpha,q}\gamma(x,t)\leq 1$ on $\text{spt}(\gamma)$ and then,  ${\textbf{W}}_{\alpha,q}\gamma (x,t)\leq c$. 
Hence, we have 
\begin{equation}
	\text{spt}(\gamma)\subset 
	\left\{ (x,t): [M_{\gamma}\mu(x,t)]^{q'-1}> \vartheta/c\right\}=\left\{ (x,t): M_{\gamma}\mu(x,t) >(\vartheta/c)^{1/(q'-1)} \right\}\nonumber.
\end{equation}
This show us $\frac{\mu(Q_{5r}(x,t))}{\gamma(Q_r(x,t))}>(\vartheta/c)^{1/(q'-1)}$ on 
$\text{spt}(\gamma)$. Let $\mathscr{F}=\{Q_{r_i}(z_i)\,:\, z_i\in\text{spt}(\gamma) \text{ and } i\in J\}$ be the family of backward parabolic balls that cover $K\supset \text{spt}(\gamma)$ and such that $\frac{\mu(Q_{5r_i}(z_i))}{\gamma(Q_{r_i}(z_i))}>(\vartheta/c)^{1/(q'-1)}$. Since $K$ is compact and $(\mathbb{R}^{d+1}, d_{\mathcal{P}}, d\mathscr{L})$ is a homogeneous space, 
then by Vitali Covering Theorem  (see e.g. \cite[Theorem 1.2]{CW} or \cite[Lemma 2.1]{Aimar}), there is a finite subfamily $\mathscr{F}_I\subset \mathscr{F}$ of pairwise disjoint backward balls $\{Q_j=Q_{r_j}(z_j)\}_{j\in I}$ such that $\{5Q_j\}_{j\in I}$ covers $K\supset \text{spt}(\gamma)$ and  
\begin{equation}
	\frac{\mu(5Q_j)}{\gamma(Q_j)}>{(\vartheta/c)}^{1/(q'-1)}.\nonumber
\end{equation}
Therefore,
\begin{equation}
	\mathscr{C}_{\alpha,q}(K)=\gamma(K)\leq \sum_{j} \gamma(Q_j)< \left(\frac{c}{\vartheta}\right)^{q-1}\sum_{j} \mu(5Q_j)\leq \vartheta^{1-q}\vert \mu\vert(K)\nonumber,
\end{equation}
this finish the proof  as desired. 
\end{proof}

\bigskip 

Now we effectively achieve our objective.
\begin{theorem}\label{Conjuntos finos Geral parabolico}The set $E \subset \mathbb{R}^{d+1}$ is parabolic $(\alpha,q)-$thin at $ z_0 \in \overline{E}$ if and only if there exists $\mu \in \mathfrak{M}^{+}$ such that
\begin{equation}\label{equiv-thin2}
	{\textnormal{\textbf{W}}}_{\alpha,q} \mu (z_0)< \liminf_{\substack{z \to z_0 \\ z \in E \setminus \{z_0\}}}{\textnormal{\textbf{W}}}_{\alpha,q} \mu( z).
\end{equation}
\end{theorem}

\begin{proof}The inequality  \eqref{equiv-thin2} is equivalent to the condition:  there exists a backward parabolic ball $B_R=Q_R(z_0)$ and a positive measure $\mu\in\mathfrak{M}^{+}(\mathbb{R}^{d+1})$ such that 
	\begin{itemize}
		\item[\textit{(i)}]  $\dWolff \mu (z_0)<\infty$,
		\item[\textit{(ii)}]$		\dWolff \mu(z)\geq \dWolff\mu (z_0)+\eta \quad \text{ for all }\quad  z\in E\cap B_R\backslash\{ z_0\} $.
	\end{itemize} 
	for some  $\eta>0$. We will show that last condition implies that $E$ is parabolic $(\alpha,q)-$thin at $ z_0$. Initially we will construct a measure $\gamma\in\mathfrak{M}^+$ supported on $B_R$ such that 
	\begin{equation}
		\dWolff \gamma (z_0)<\varepsilon \quad \text{ and }\quad \dWolff\gamma(z)\geq 1  \quad \forall\, z\in E\cap B_R\backslash\{ z_0\},\nonumber
	\end{equation} 
	for all $\varepsilon>0$. To this end, let $\mu_R=\mu_{\lfloor B( z_0,R)}$ be the restriction of $\mu$ to the  backward ball $B(z_0,R)$. Choosing $z\in B( z_0,R/2)$, we have $B(z,\delta)\subset B(z_0,R) $ if provided that  $\delta<R/2$. Then,  $\mu_R(B(z,\delta))=\mu(B(z,\delta))$ which leads to 
	\begin{equation}
		\dWolff \mu(z)-\dWolff\mu_R(z)=\int_{{R}/{2}}^{1}\frac{\mu(B(z,\delta))^{q'-1} -\mu_R(B(z,\delta))^{q'-1}}{\delta^{(n-\alpha q)(q'-1)}}\frac{d\delta}{\delta}.\nonumber
	\end{equation}
	The above expression is a continuous function  at $z$. Therefore, for any $\epsilon>0$ there exists a backward ball $B_R(z_0)$ such that  
	\begin{align}
	\dWolff \mu_R (z)-\dWolff \mu_R(z_0)&\geq \dWolff \mu (z)-\dWolff \mu(z_0)-\epsilon\geq \eta -\epsilon \quad \nonumber,
	\end{align}
	for all $z\in E\cap B_R(z_0)\backslash \{ z_0\} $. In particular, we have 
	\begin{align}
		\dWolff \mu_R (z)\geq \dWolff \mu_R  (z_0)+\eta \quad \forall \, z\in B_R(z_0)\cap E\backslash \{ z_0\}.\label{thin-equiv}
	\end{align}
	Note that $\dWolff \mu_{R} (z_0)\rightarrow 0$ as $R\rightarrow0$. Then, we can choose $R>0$ sufficiently small such that $\dWolff \mu_R (z_0)<\eta \varepsilon$ for all $\varepsilon>0$. Thus, for  $\gamma =\eta^{1-q}\mu_R$,  
	\begin{equation}
		\dWolff \gamma (z_0)=\eta^{(1-q)(q'-1)}\dWolff\mu_R (z_0)<\varepsilon. \label{control-epsilon}
	\end{equation}
	Moreover, by inequality (\ref{thin-equiv}),  
	\begin{align}
		\dWolff\gamma (z)=\eta^{(1-q)(q'-1)}\dWolff\mu_R(z)&\geq \eta^{(1-q)(q'-1)}\left( \dWolff\mu_R (z_0)+\eta\right)\nonumber\\
		&=\dWolff\gamma (z_0)+1\geq 1,\label{Wolff>1}
	\end{align}
	for $R>0$ sufficiently small and $z\in B_R(z_0)\cap E\backslash \{ z_0\}$. For each fixed $\rho<R$, let  $\gamma_{\rho}=\gamma_{\lfloor B( z_0,\rho)}$ denote  the restriction of $\gamma$ on backward ball $B(z_0,\rho)$ and $\gamma'_{\rho}=\gamma_{\lfloor B( z_0,\rho)^{c}}$ the restriction to the complement of $B(z_0,\rho)$. Assume that $\dWolff \gamma_{\rho} (z)>1/4 A$ on  $E\cap B( z_0,\rho/2)\backslash\{ z_0\}$, where  $A>0$ is a constant.   Then, Proposition \ref{infinity-energy} yields 
	\begin{equation}
		\mathscr{C}_{\alpha,q}(E\cap B( z_0,\rho/2))\leq \mathscr{C}_{\alpha,q}(\{\dWolff\gamma_{\rho}(z)>1/4A\})\leq c\gamma (B( z_0,\rho)). \nonumber
	\end{equation}
	Consequently, 
	\begin{equation}
		\int_{0}^{1}\left(\frac{\mathscr{C}_{\alpha,q}(E \cap B( z_0,\rho/2))}{\rho^{n-\alpha q}}\right)^{q'-1}\frac{d\rho}{\rho}\leq \int_{0}^{1}\left(\frac{\gamma(B( z_0,\rho))}{\rho^{n-\alpha q}}\right)^{q'-1}\frac{d\rho}{\rho}<\infty,\nonumber
		\end{equation}
	as desired. It remains to show that $\dWolff \gamma_{\rho} (z)>1/4 A$ for   $z\in E\cap B( z_0,\rho/2)\backslash\{ z_0\}$. To this end, we establish firsly that 
    \[
    \dWolff \gamma'_{\rho}(z) \leq c\, \dWolff \gamma(z_0), \quad \text{for} \quad z \in B(z_0, \rho/2).
    \]
     Indeed, note that 
	\begin{align}
		\dWolff\gamma'_{\rho}(z) &=  \int_{\rho/2}^{1} \left(\frac{\gamma'_{\rho}(B(z,\delta))}{\delta^{n-\alpha q}}\right)^{q'-1}\frac{d\delta}{\delta}\nonumber\\
		&\leq \int_{\rho/2}^{1} \left(\frac{\gamma(B(z,\delta))}{\delta^{n-\alpha q}}\right)^{q'-1}\frac{d\delta}{\delta}\nonumber\\
		&=\int_{\rho/2}^{2R} \left(\frac{\gamma(B(z,\delta))}{\delta^{n-\alpha q}}\right)^{q'-1}\frac{d\delta}{\delta}+\int_{2R}^{1} \left(\frac{\gamma(B(z,\delta))}{\delta^{n-\alpha q}}\right)^{q'-1}\frac{d\delta}{\delta}\nonumber\\
		&=I_1+I_2,\label{I_1,I_2}
	\end{align}
	since for $\delta<{\rho}/{2}$ one has  $B(z,\delta)\cap [B( z_0,\rho)]^c=\emptyset$. Let us estimate $I_2$. For any $z\in B(z_0,\rho/2)$, it follows that $B(z,\delta)\subset B(z_0,2\delta)$, if provided that $\delta/2>R>\rho$.  
	Therefore,  
	\begin{align}
		I_2&\leq \int_{2R}^{1} \left(\frac{\gamma(B( z_0,2\delta))}{\delta^{n-\alpha q}}\right)^{q'-1}\frac{d\delta}{\delta}=c\int_{4R}^{1} \left(\frac{\gamma(B( z_0,r))}{r^{n-\alpha q}}\right)^{q'-1}\frac{dr}{r}\leq c \dWolff\gamma (z_0).\nonumber
	\end{align}
	Similarly, if $\rho/2<\delta<2R$ then $B(z,\delta)\subset B(z_0,2\delta)$ which leads to 
	\begin{align}
		I_1&\leq c\int_{\rho}^{4R} \left(\frac{\gamma(B( z_0,r))}{r^{n-\alpha q}}\right)^{q'-1}\frac{dr}{r}\leq c \dWolff\gamma (z_0).\nonumber
	\end{align}
	Inserting the previous estimates into \eqref{I_1,I_2} we get the desired control. Then, inequality \eqref{control-epsilon} leads to 
	\begin{equation}
		\dWolff\gamma'_{\rho}(z)\leq c\dWolff\gamma (z_0)<c\varepsilon.\label{control-gamma'}
	\end{equation} 
	Since $\dWolff \gamma(z)\leq A(\dWolff\gamma_{\rho}(z)+\dWolff\gamma'_{\rho}(z))$, from  \eqref{control-gamma'}  and \eqref{Wolff>1} we derive 
	\begin{align}
		\dWolff\gamma_{\rho}(z)\geq \frac{1}{A}\dWolff\gamma(z)-\dWolff\gamma'_{\rho}(z)
		&> \frac{1}{A}-c\varepsilon> \frac{1}{4A}\quad \text{on}\quad E\cap B( z_0,\rho/2)\backslash \{z_0\} \nonumber,
	\end{align}
	chossing $\varepsilon$ and $\rho$ sufficiently small. 
	
	Now assume that $E$ is parabolic $(\alpha,q)-$thin at $ z_0$. It follows by Lemma \ref{lemao}(ii) that, for every $\varepsilon>0$, there exists a backward ball $T=B_{r} (z_0)$ and a capacitary measure $\gamma\in\mathfrak{M}^+$ for $E\cap T$ such that
		\begin{equation}
			{\mathbf{W}}_{\alpha,q}\gamma (z_0)<\varepsilon \quad \text{ and }\quad  {\mathbf{W}}_{\alpha,q}\gamma(z)\geq 1 \quad {\mathscr{C}}_{\alpha,q}-a.e.\quad z\in E\cap T.\nonumber
		\end{equation}
		In particular,  there exists $\eta>0$ sufficiently small such that 
		\begin{align}
			{\mathbf{W}}_{\alpha,q}\gamma(z)\geq {\mathbf{W}}_{\alpha,q}\gamma (z_0)+\eta \quad {\mathscr{C}}_{\alpha,q}-a.e.\quad E\cap T. \nonumber
		\end{align}
		We need to show there exist a measure $\mu\in\mathfrak{M}^+(\mathbb{R}^{d+1})$ and a constant $c>0$ such that 
		\begin{equation}\label{Thm6-result2}
			\dWolff \mu (z) \gtrsim  \dWolff \mu  (z_0) +\frac{\eta}{2} \quad \text{for all} \quad z\in E\cap T\backslash \{z_0\}.
		\end{equation}
Indeed, let $e=\{z\in E\cap T\,:\, \dWolff\gamma(z)<{\mathbf{W}}_{\alpha,q}\gamma (z_0)+\eta\}$ and  assume there exists a  measure $\upsilon\in\mathfrak{M}^+$ such that 
	\begin{align}
	&\dWolff \upsilon  (z) \gtrsim \dWolff \gamma (z_0) +\eta \quad \text{ for all }\quad z\in e\backslash \{z_0\}\label{Thm6-ineqv1}\\
	&\dWolff \upsilon  (z_0) <\frac{\eta}{2}.\label{Thm6-ineqv2}
	\end{align}
	Then, the measure $\mu=\gamma+\upsilon$ satisfies \eqref{Thm6-result2}. Indeed, by $\dWolff \mu  (z_0)\leq c(\dWolff \gamma (z_0)+\dWolff\upsilon  (z_0))$ we have for $z\in e\backslash \{z_0\}$, 
	\begin{align*}
	\dWolff \mu (z)\geq \dWolff \upsilon (z) &\gtrsim \dWolff \gamma(z_0) +\eta\\
	&\geq \frac{1}{c}\dWolff \mu (z_0) - \dWolff\upsilon (z_0) +\eta\\ 
	&= \left( \frac{1}{c}\dWolff \mu (z_0) +\frac{\eta}{2}\right)+\left( \frac{\eta}{2}-\dWolff\upsilon (z_0)\right)\\
	&> \frac{1}{c}\dWolff \mu (z_0) +\frac{\eta}{2}.
	\end{align*}
Similarly, for $z\in E\cap T\backslash (e\cup \{z_0\})$, 
	\begin{align*}
		\dWolff \mu (z)\geq \dWolff \gamma (z) \geq \dWolff \gamma(z_0) +\eta > \frac{1}{c}\dWolff \mu (z_0) +\frac{\eta}{2}.
		\end{align*}
		It remains to verify  \eqref{Thm6-ineqv1} and \eqref{Thm6-ineqv2}. For each $j\in\mathbb{N}$, set $e_j=e\cap \mathcal{A}_j$, where $\mathcal{A}_j=B( z_0,2^{-j+1}r)\backslash B( z_0, 2^{-j}r)$, and let $\lambda_j$ be a $\mathscr C$-capacitary measure for $e_j$. Define 
		\begin{equation}
		\upsilon =(\dWolff\gamma(z_0)+\eta)^{q-1}\lambda, \quad \text{with} \quad  \lambda=\sum_j \lambda_j.\nonumber
		\end{equation}
Then $\upsilon\in\mathfrak{M}^+$ satisfies \eqref{Thm6-ineqv1} and \eqref{Thm6-ineqv2}. Indeed,  for $z\in e\backslash\{z_0\}$,
		\begin{align}
			\dWolff \upsilon (z)= \left(\dWolff\gamma(z_0)+\eta\right)^{(q-1)(q'-1)}\dWolff \lambda(z)&\geq  \left(\dWolff\gamma(z_0)+\eta\right)\dWolff \lambda_j(z)\nonumber\\
			&\geq \dWolff\gamma(z_0)+\eta\nonumber.
		\end{align}
		Since $\mathscr{C}_{\alpha,q}(e)=0$, Lemma \ref{lemao}(ii) permits the choice of each $\lambda_j$ so that $\mathbf{W}_{\alpha,q}\lambda_j(z_0)<\varepsilon_j$. Consequently, 
		\begin{align}
			\dWolff \upsilon  (z_0)\leq c (\dWolff\gamma(z_0)+\eta)\sum_{j}\dWolff\lambda_j (z_0)\leq  c (\dWolff\gamma(z_0)+\eta)\sum_{j}\varepsilon_j.\nonumber
		\end{align}
		Choosing $\varepsilon_j>0$ such that $\sum_{j=1}^{\infty}\varepsilon_j\leq \eta/c (\dWolff\gamma(z_0)+\eta)$, we obtain \eqref{Thm6-ineqv2}, completing the proof. 
	\end{proof}

\section{Conflict of interest}
On behalf of all authors, the corresponding author states that there is no conflict of interest.


\begin{thebibliography}{999}
	\bibitem{AH}D.~R. Adams\ and\ L.~I. Hedberg, {\it Function spaces and potential theory}, Grundlehren der mathematischen Wissenschaften, 314, Springer, Berlin, 1996.



    \bibitem{Aimar}H.~A. Aimar, Elliptic and parabolic BMO and Harnack's inequality, Trans. Amer. Math. Soc. {\bf 306} (1988), no.~1, 265--276. 

	\bibitem{Aikawa}H. Aikawa\ and\ M.~R. Ess\'{e}n, {\it Potential theory---selected topics}, Lecture Notes in Mathematics, 1633, Springer, Berlin, 1996. 


    \bibitem{Brzezina}M. Brzezina, On the base and the essential base in parabolic potential theory, Czechoslovak Math. J. {\bf 40(115)} (1990), no.~1, 87--103.

	\bibitem{Brelot}M. Brelot, Sur les ensembles effil\'es, Bull. Sci. Math. (2) {\bf 68} (1944), 12--36.


	\bibitem{CW}R.~R. Coifman\ and\ G.~L. Weiss, {\it Analyse harmonique non-commutative sur certains espaces homog\`enes}, Lecture Notes in Mathematics, Vol. 242, Springer, Berlin, 1971. 
   
	\bibitem{COV}C. Cascante, J. Ortega\ and\ I.~E. Verbitsky, Nonlinear potentials and two weight trace inequalities for general dyadic and radial kernels, Indiana Univ. Math. J. {\bf 53} (2004), no.~3, 845--882. 

	\bibitem{DiBenedetto}E. DiBenedetto, {\it Degenerate parabolic equations}, Universitext, Springer, New York, 1993.

	\bibitem{EG}L.~C. Evans\ and\ R.~F. Gariepy, Wiener's criterion for the heat equation, Arch. Rational Mech. Anal. {\bf 78} (1982), no.~4, 293--314.

	\bibitem{FR}E.~B. Fabes\ and\ N. Rivi\`ere, Singular integrals with mixed homogeneity, Studia Math. {\bf 27} (1966), 19--38.

    \bibitem{FGL} E.~B. Fabes, N. Garofalo\ and\ E. Lanconelli, Wiener's criterion for divergence form parabolic operators with $C^1$-Dini continuous coefficients, Duke Math. J. {\bf 59} (1989), no.~1, 191--232. 
    
	\bibitem{Ziemer2}R.~F. Gariepy\ and\ W.~P. Ziemer, Thermal capacity and boundary regularity, J. Differential Equations {\bf 45} (1982), no.~3, 374--388.

\bibitem{Garofolo1} N. Garofalo\ and\ E. Lanconelli, Wiener's criterion for parabolic equations with variable coefficients and its consequences, Trans. Amer. Math. Soc. {\bf 308} (1988), no.~2, 811--836.

\bibitem{Hu}X. Hu, L. Tang, Wiener's criterion for degenerate parabolic equations, arXiv:2303.08350.

	\bibitem{Hedberg}L.~I. Hedberg, Non-linear potentials and approximation in the mean by analytic functions, Math. Z. {\bf 129} (1972), 299--319. 

	\bibitem{Wolff}L.~I. Hedberg\ and\ T.~H. Wolff, Thin sets in nonlinear potential theory, Ann. Inst. Fourier (Grenoble) {\bf 33} (1983), no.~4, 161--187.

	\bibitem{Jones}B.~F. Jones Jr., Lipschitz spaces and the heat equation, J. Math. Mech. {\bf 18} (1968/69), 379--409.

\bibitem{Mingione}T. Kuusi and G. Mingione, Riesz potentials and nonlinear parabolic equations, Arch. Ration. Mech. Anal. {\bf 212} (2014), no.~3, 727--780.

\bibitem{Mingione2}T. Kuusi and G. Mingione, Gradient regularity for nonlinear parabolic equations, Ann. Sc. Norm. Super. Pisa Cl. Sci. (5) {\bf 12} (2013), no.~4, 755--822

 \bibitem{Landkof}N.~S. Landkof, {\it Foundations of modern potential theory}, translated from the Russian by A. P. Doohovskoy, 
Die Grundlehren der mathematischen Wissenschaften, Band 180, Springer, New York-Heidelberg, 1972.

\bibitem{Landis}E. Landis, Necessary and sufficient conditions for the regularity of a boundary point for the Dirichlet problem for the heat equation, Dokl. Akad. Nauk SSSR {\bf 185} (1969), 517--520.

\bibitem{Lanconelli}E. Lanconelli, Sul problema di Dirichlet per l'equazione del calore, Ann. Mat. Pura Appl. (4) {\bf 97} (1973), 83--114. 

\bibitem{Lerner}A.~K. Lerner and F.~L. Nazarov, Intuitive dyadic calculus: the basics, Expo. Math. {\bf 37} (2019), no.~3, 225--265

\bibitem{Meyers}N.~G. Meyers, A theory of capacities for potentials of functions in Lebesgue classes, Math. Scand. {\bf 26} (1970), 255--292 (1971). 

\bibitem{HM}V.~G. Maz'ya and V.~P. Khavin, A nonlinear potential theory, Uspehi Mat. Nauk {\bf 27} (1972), no.~6, 67--138.

\bibitem{Nguyen}Q.-H. Nguyen, Potential estimates and quasilinear parabolic equations with measure data, Mem. Amer. Math. Soc. {\bf 291} (2023), no.~1449, v+123 pp.
            
\bibitem{Sampson}C.~H. Sampson, {\it a characterization of parabolic Lebesgue spaces}, Thesis (Ph.D.) Rice University, ProQuest LLC, Ann Arbor, MI, 1968. 

	\bibitem{Saari}O. Saari, Parabolic BMO and the forward-in-time maximal operator, Ann. Mat. Pura Appl. (4) {\bf 197} (2018), no.~5, 1477--1497. 

	\bibitem{Segovia-Wheeden}C. Segovia~Fern\'{a}ndez\ and\ R.~L. Wheeden, On the function $g\sb{\lambda }\sp{\ast} $ and the heat equation, Studia Math. {\bf 37} (1970), 57--93.

	\bibitem{Tralli}G. Tralli\ and\ F. Uguzzoni, A Wiener test \`a la Landis for evolutive H\"{o}rmander operators, J. Funct. Anal. {\bf 278} (2020), no.~6, 108410, 34 pp.




    \bibitem{Ziemer1}W.~P. Ziemer, Regularity of weak solutions of parabolic variational inequalities, Trans. Amer. Math. Soc. {\bf 309} (1988), no.~2, 763--786.

	\bibitem{Wiener} N.~G. Wiener, Discontinuous boundary conditions and the Dirichlet problem, Trans. Amer. Math. Soc. {\bf 25} (1923), no.~3, 307--314.

	\bibitem{Watson}N.~A. Watson, {\it Introduction to heat potential theory}, Mathematical Surveys and Monographs, 182, Amer. Math. Soc., Providence, RI, 2012.

    
\end{thebibliography}
\end{document}